\documentclass{article}

\usepackage{amssymb}
\usepackage{amsmath}
\usepackage{amsfonts}
\usepackage{amsthm}
\usepackage{caption,subcaption}
\usepackage{mathabx}
\usepackage{graphicx}
\usepackage{cite}
\usepackage{tikz}
\usepackage{placeins}
\usepackage{float}
\usepackage{pstricks}
\usepackage{tikz-cd}

\usepackage[breaklinks=true]{hyperref}

\usetikzlibrary{arrows,chains,matrix,positioning,scopes}

\newtheorem{thm}{Theorem}[section]
\newtheorem{theorem}[thm]{Theorem}
\newtheorem{lem}[thm]{Lemma}

\newtheorem{prop}[thm]{Proposition}
\newtheorem{cor}[thm]{Corollary}

\newtheorem{thmm}{Theorem}

\newtheorem{innercustomthm}{Theorem}
\newenvironment{customthm}[1]
  {\renewcommand\theinnercustomthm{#1}\innercustomthm}
  {\endinnercustomthm}

\theoremstyle{definition}
\newtheorem{definition}[thm]{Definition}

\newtheorem{notation}[thm]{Notation}
\newtheorem*{unnumberednotation}{Notation}

\theoremstyle{remark}
\newtheorem{rem}{Remark}

\newtheorem*{prop proof}{Proof of Proposition}

\newcommand{\RR}{\mathbb{R}}      
\newcommand{\ZZ}{\mathbb{Z}}        
\newcommand{\NN}{\mathbb{N}}      
\newcommand{\EE}{\mathbb{E}}
\newcommand{\lk}{{\rm{Lk}}}

\newcommand{\abs}[1]{\left\lvert#1\right\rvert}

\numberwithin{equation}{section}

\title{Hyperbolic Groups with Finitely Presented Subgroups not of Type $F_3$}

\author{Robert Kropholler \\ with an appendix by Giles Gardam}

\begin{document}
\maketitle

\begin{abstract}
We generalise the constructions in \cite{brady_branched_1999} and \cite{lodha_finiteness_2017} to give infinite families of hyperbolic groups, each having a finitely presented subgroup that is not of type $F_3$. By calculating the Euler characteristic of the hyperbolic groups constructed, we prove that infinitely many of them are pairwise non isomorphic. We further show that the first of these constructions cannot be generalised to dimensions higher than $3$.
\end{abstract}

\tableofcontents

\section{Introduction}

In this paper we look at subgroups of hyperbolic groups: these are not in general hyperbolic. For instance, one can take a non-trivial infinite index normal subgroup of a free group: this is not finitely generated. For elaborate examples see \cite{brady_branched_1999} and \cite{lodha_finiteness_2017}. In \cite{brady_branched_1999} there is an example of a graph $\Theta$ and a branched cover of $\Theta^3$ which has hyperbolic fundamental group. This group is shown to have a subgroup which is finitely presented but not hyperbolic. We show that one can change the graph $\Theta$ as long as a certain condition, 2-fullness, is preserved. We establish a generalisation of this result by proving the following theorem.

\begin{thmm}
Let $\Gamma_1$, $\Gamma_2$ and $\Gamma_3$ be 2-full graphs where every vertex has valence at least 4. There is a branched cover of $\Gamma_1\times\Gamma_2\times\Gamma_3$ whose fundamental group is hyperbolic and contains a finitely presented subgroup that is not hyperbolic. 
\end{thmm}

A variation on this theme can be found in \cite{lodha_finiteness_2017}. There, another example of a finitely presented non-hyperbolic subgroup is obtained by considering a subcomplex of $K\times K\times K$ where $K$ is a complete bipartite graph. This construction relies on a graph $\Gamma$ with certain properties. We prove the following theorem which allows us to relax some of these properties and create smaller examples. 

\begin{thmm}
Let $\Gamma_1$, $\Gamma_2$ and $\Gamma_3$ be sizeable graphs with vertex sets $A_i\cup B_i, i=1, 2, 3$ and let $K_{ij}$ be the complete bipartite graph on $A_i$ and $B_j$. Then there exists $X\subset K_{13}\times K_{21}\times K_{32}$ such that $\pi_1(X)$ is hyperbolic and has a finitely presented subgroup that is not hyperbolic.
\end{thmm}

The terminology we use in these theorems is defined in sections 3 and 4.  

Throughout the paper we will use the failure of finiteness properties of groups to prove that some finitely presented subgroups of hyperbolic groups are not hyperbolic. 

\begin{definition}
A group $G$ satisfies {\em property $F_n$} if there is a $K(G, 1)$ with finite $n$ skeleton.
\end{definition}

$F_1$ is equivalent to being finitely generated and $F_2$ is equivalent to being finitely presented. 

\begin{definition}
A group $G$ is of {\em type $FP_n$} if there is a partial resolution of the trivial $\ZZ G$ module, 

\begin{figure*}[h]
\center
\begin{tikzpicture}[node distance=1.2cm, auto]
\node (Z) {$\ZZ$};
\node (P_0) [left of=Z] {$P_0$};
\node (P_1) [left of=P_0] {$P_1$};
\node (dots) [left of=P_1] {$\dots$};
\node (P_n) [left of=dots] {$P_n$};
\node (0r) [right of=Z] {$0$}; 

\draw[->,] (P_1) to node {} (P_0);
\draw[->,] (P_0) to node {} (Z);
\draw[->,] (dots) to node {} (P_1);
\draw[->,] (P_n) to node {} (dots);
\draw[<-,] (0r) to node {} (Z);
\end{tikzpicture}
\end{figure*}
where $P_i$ is a finitely generated projective $\ZZ G$ module.
\end{definition}

If $G$ is of type $F_n$, then $G$ is of type $FP_n$.

A theorem of Rips (for details see \cite[III.H.3.21]{bridson_metric_1999}) tells us that hyperbolic groups are $F_n$ for all $n$

\begin{theorem}[Rips]
Let $G$ be a hyperbolic group. Then there is a contractible simplicial complex $K$, known as the Rips complex, and a proper, cocompact action of $G$ on $K$.
\end{theorem} 

Rips gave the first examples of subgroups of hyperbolic groups which are finitely generated $(F_1)$ but not finitely presented $(F_2)$. Using small cancellation theory, he obtained the following.  

\begin{theorem}[Rips,\cite{rips_subgroups_1982}, Theorem 1]
For every finitely presented group $G$ there exists a hyperbolic group $H$ which surjects onto $G$, such that the kernel $K$ is finitely generated. If $G$ is infinite, then $K$ is not finitely presented.
\end{theorem}

Once this was known it was natural to ask whether finitely presented subgroups of hyperbolic groups are hyperbolic. A positive result in this direction was given by Gersten. 

\begin{thm}[Gersten, \cite{gersten_subgroups_1996}, Theorem 5.4]
If $G$ is a hyperbolic group of cohomological dimension 2 and $H < G$ is finitely presented, then $H$ is hyperbolic. 
\end{thm}

However, we have already seen that there are examples of hyperbolic groups with subgroups that are finitely presented but not hyperbolic. In the examples from \cite{brady_branched_1999} and \cite{lodha_finiteness_2017}, as well as our Theorems \ref{bradygen} and \ref{lodhagen}, the groups have cohomological dimension 3. In the final section of this paper we show that the first of the two techniques is very special to dimension 3 and prove the following:

\begin{thmm}
Let $X$ be a product of more than $3$ graphs. Then no branched cover of $X$ is hyperbolic. 
\end{thmm}

One would ideally like to extend the results of this paper to construct subgroups of hyperbolic groups that are of type $F_n$ but not $F_{n+1}$, or an example of a subgroup with a finite classifying space which is not hyperbolic. The techniques in this paper will not help us with these questions. One of the key difficulties in regard to these challenges is the paucity of techniques for constructing hyperbolic groups of high cohomological dimension.

\subsection*{Acknowledgements}

The author thanks Martin Bridson for helpful comments on this paper. The author was partially supported by an EPSRC studentship.
The author of the appendix was partially supported by the Clarendon Fund, Balliol College Marvin Bower Scholarship, James Fairfax Oxford Australia Scholarship, and the Israel Science Foundation (grant 662/15).

\section{Preliminaries}
\subsection{Cube Complexes}

A cube complex can be constructed by taking a collection of disjoint cubes and gluing them together by isometries of their faces.

There is a standard way in which cube complexes can be endowed with metrics, and there is a well-known criterion from \cite{chern_hyperbolic_1987} that characterises those cube complexes that are locally CAT(0).

The precise definition of a cube complex is given in \cite[Def. 7.32]{bridson_metric_1999} as follows.

\begin{definition}
A {\em cube complex} $X$ is a quotient of a disjoint union of Euclidean cubes $K = \bigsqcup_{c \in C}[0,1]^{n_c}$ by an equivalence relation $\sim$. 

The restrictions $\chi_c\colon [0,1]^{n_c}\to X$ of the natural projection $\chi\colon K\to X = K/\sim$ are required to satisfy:
\begin{itemize}
\item for every $c\in C$ the map $\chi_c$ is injective;
\item if $\chi_c([0,1]^{n_c})\cap \chi_{c'}([0,1]^{n_{c'}})\neq \emptyset$, then there is an isometry $h_{c,c'}$ from a face $T_c\subset [0,1]^{n_c}$ onto a face $T_{c'}\subset [0,1]^{n_c'}$ such that $\chi_c(x) = \chi_{c'}(x')$ if and only if $x' = h_{c,c'}(x)$.
\end{itemize}
\end{definition}

\begin{definition}
A metric space is {\em non-positively curved} if its metric is locally CAT(0).
\end{definition}

We will see in due course that non-positive curvature for cube complexes is a local condition. It is controlled at the links of vertices. 

\begin{definition}
	Let $x$ be a point in a cube complex $X$. We define the {\em link of $x$}, denoted $\lk(x, X)$ to be the $\epsilon$ neighbourhood of $x$ for $\epsilon$ sufficiently small. 
	
	We define the {link of a cube $c$}, $\lk(c, X)$ to be the link of any interior point in $c$. If $c$ is an $n$-cube, then $\lk(c, X)$ is a join $\Lambda\ast S^{n-1}$.
\end{definition}

We should take $\epsilon$ smaller than the distance from $x$ to any cube not containing $x$. 

We can subdivide the cube complex to make $x$ a vertex. This link then comes with natural simplicial structure, where $n$-cells in the link are intersections of $N_\epsilon(x)$ with $n+1$-cubes in this new cubical structure. 

This link is a spherical complex built from all-right spherical simplices (see \cite[\S5.18]{bridson_metric_1999}). 

\begin{definition}
A complex $L$ is a {\em flag complex} if it is simplicial and every set $\{v_1, \dots, v_n\}$ of pairwise adjacent vertices spans a simplex. (That is, there are no ``empty simplices'').
\end{definition}

Flag complexes are completely determined by their $1$-skeleta. 

Gromov proved the following combinatorial condition for cube complexes.

\begin{thm}[Gromov, \cite{chern_hyperbolic_1987}]
A cube complex is non-positively curved if and only if the link of every vertex is a flag complex. 
\end{thm}

In the course of this paper we will construct subgroups of hyperbolic groups which are not hyperbolic. We therefore require a criterion which will tell us when the fundamental group of a compact cube complex is hyperbolic. This will be the case if the universal cover is a hyperbolic metric space. Bridson, following a suggestion of Gromov, gave a nice characterisation of this property as follows. 

\begin{thm}[Bridson, \cite{bridson_existence_1995}, Theorem A]\label{flatplane}
Let $X$ be a compact non-positively curved cube complex and let $\tilde{X}$ be its universal cover. Then $\tilde{X}$ is not hyperbolic if and only if there exists an isometric embedding $i\colon \EE^2\hookrightarrow \tilde{X}$. 
\end{thm}

We call such an isometrically embedded copy of $\EE^2$ a {\em flat}.

The key arguments in our proofs will be proving that the complexes constructed have no flat planes in their universal covers. To do this we will need an appropriate form of transvervality in the world of cube complexes; we define this as follows.

\begin{definition}
Given a CAT(0) cube complex $X$ and an isometric embedding $i\colon \EE^2\to X$, we say that a subset $D$ of $X$ intersects $\EE^2$ {\em transversally} at a point $p$ if there is an $\epsilon>0$ such that $N_{\epsilon}(p)\cap D\cap i(\EE^2) = \{p\}$.
\end{definition}

\subsubsection{CAT(0) cones}

We will consider flats in CAT(0) cube complexes. For every point $x\in X$ on the flat, there is a corresponding geodesic loop of length $2\pi$ in ${\rm Lk}(x, X)$. In many cases the link of such a point will be a spherical join, and it will be natural to consider the Euclidean cone on this link and examine flat planes in this cone. For full details see \cite[I.5.13]{bridson_metric_1999}

\begin{definition}\label{conedef}
Given a metric space $Y$, the {\em CAT(0) cone} $X = C_0(Y)$ over $Y$ is defined as follows. As a set $X$ is a quotient of $[0,\infty)\times Y$ by the equivalence relation given by $(t, y)\sim (t',y')$ if $(t=t'=0)$. The equivalence class of $(t,y)$ is denoted $ty$ and the class of $(0,y)$ is denoted $0$.

The distance between two points $x = ty$ and $x' = t'y'$ in $X$ is defined by 
$$d(x,x')^2 = t^2 + t'^2 - 2tt'\cos\big(\min\{\pi,d(y,y')\}\big).$$
This formula defines a metric on $X$ see \cite[I.5.9]{bridson_metric_1999}.
\end{definition}

\begin{rem}
The metric defined above is CAT(0) if and only if $Y$ is a CAT(1) space. 
\end{rem}

When considering cones, we have the following theorem telling us that joins of metric spaces correspond to products in the cone.

\begin{thm}[Bridson--Haefliger \cite{bridson_metric_1999}, I.5.15]\label{joinisproduct}
For any metric spaces $Y_1$ and $Y_2$ there is a natural isometry of $C_0(Y_1\ast Y_2)$ onto $C_0(Y_1)\times C_0(Y_2)$.
\end{thm}

In this setting, we consider the projection of $C_0(Y_1)\times C_0(Y_2)$ onto each factor. These projections do not increase distances and (after reparameterisation) map geodesics to geodesics. 

Two geodesic lines $c, c'\colon \RR\to X$ are said to be {\em asymptotic} if there exists a $k\in\RR$ such that $d\big(c(t), c'(t)\big)\leq k$ for all $t\in\RR$. Asymptotic rays in CAT(0) spaces behave very nicely:

\begin{thm}[Bridson--Haefliger \cite{bridson_metric_1999}, II.2.13]\label{flatstrip}
Let $X$ be a CAT(0) space and $c, c'$ be asymptotic geodesic lines. Then the convex hull of $c(\RR)\cup c'(\RR)$ is isometric to a flat strip $$\RR\times [0,D]\subset\EE^2.$$ 
\end{thm}

\subsection{Branched covers of cube complexes}

We will take branched covers of cube complexes to get rid of high dimensional flats. The techniques we will use were developed by Brady in \cite{brady_branched_1999}. The idea is to branch over an appropriate subset which intersects all the high-dimensional flats.

\begin{definition}\label{locus}
Let $K$ be a non-positively curved cube complex. We say that $L\subset K$ is a {\em branching locus} if it satisfies the following conditions:
\begin{enumerate}
\item $L$ is a locally convex cubical subcomplex, 
\item ${\rm Lk}(c, K)\smallsetminus L$ is connected and non-empty for all cubes $c$ in $L$.
\end{enumerate}
\end{definition}

The first condition is required to prove that non-positive curvature is preserved when taking branched covers. The second is a reformulation of the classical requirement that the branching locus has codimension 2 in the theory of branched covers of manifolds; it ensures that the trivial branched covering of $K$ is $K$. 

\begin{definition}
A {\em branched cover} $\widehat{K}$ of $K$ over the branching locus $L$ is the result of the following process.
\begin{enumerate}
\item Take a finite covering $\overline{K\smallsetminus L}$ of $K\smallsetminus L$.
\item Lift the piecewise Euclidean metric locally and consider the induced path metric on $\overline{K\smallsetminus L}$.
\item Take the metric completion $\widehat{K}$ of $\overline{K\smallsetminus L}$.
\end{enumerate}
\end{definition}

We require some key results from \cite{brady_branched_1999} which allow us to conclude that this process is natural and that the resulting complex is still a non-positively curved cube complex. 

\begin{lem}[Brady \cite{brady_branched_1999}, Lemma 5.3]
There is a natural surjection $\widehat{K}\to K$ and $\widehat{K}$ is a piecewise Euclidean cube complex. 
\end{lem}

\begin{lem}[Brady \cite{brady_branched_1999}, Lemma 5.5]\label{npcbranch}
If $L$ is a finite graph, then $\widehat{K}$ is non-positively curved. 
\end{lem}

\subsection{Bestvina--Brady Morse theory}

While Bestvina--Brady Morse theory is defined in the more general setting of affine cell complexes, in this instance we shall only need it for non-positively curved cube complexes. 

For the remainder of this section, let $X$ be a CAT(0) cube complex and let $G$ be a group which acts freely, cellularly, properly and cocompactly on $X$. Let $\phi\colon G\to\ZZ$ be a homomorphism and let $\ZZ$ act on $\RR$ by translations.

Recall that $\chi_c$ is the characteristic map of the cube $c$.

\begin{definition}
We say that a function $f\colon X\to \RR$ is a {\em $\phi$-equivariant Morse function} if it satisfies the following 3 conditions.
\begin{itemize}
\item For every cube $c\subset X$ of dimension $n$, the map $f\chi_c\colon [0,1]^n\to\RR$ extends to an affine map $\RR^n\to\RR$ and $f\chi_c\colon [0,1]^n\to\RR$ is constant if and only if $n=0$.
\item The image of the $0$-skeleton of $X$ is discrete in $\RR$.
\item $f$ is $\phi$-equivariant, that is, $f(g\cdot x) = \phi(g)\cdot f(x)$.
\end{itemize}
\end{definition}

We will consider the level sets of our function, which we will denote as follows. 

\begin{definition}
For a non-empty closed subset $I\subset\RR$ we denote by $X_I$ the preimage of $I$. We also use $X_t$ to denote the preimage of $t\in\RR$.
\end{definition}

The kernel $H$ of $\phi$ acts on the cube complex $X$ in a manner preserving each level set $X_{I}$. Moreover, it acts properly and cocompactly on the level sets. We will use the topological properties of the level sets to gain information about the finiteness properties of the group. We will need to examine how they vary as we pass to larger level sets.

\begin{thm} [Bestvina--Brady, \cite{bestvina_morse_1997}, Lemma 2.3]
If $I\subset I'\subset\RR$ are connected and $X_{I'}\smallsetminus X_{I}$ contains no vertices of $X$, then the inclusion $X_I\hookrightarrow X_{I'}$ is a homotopy equivalence. 
\end{thm}

If $X_{I'}\smallsetminus X_{I}$ contains vertices of $X$, then the topological properties of $X_{I'}$ can be very different from those of $X_I$. This difference is encoded in the ascending and descending links. 

\begin{definition}
The {\em ascending link} of a vertex is 

$Lk_{\uparrow}(v,X) = \bigcup \{{\rm Lk}(w,c)\mid c$ is a cube in $X$ and $\chi_c(w) = v$ and $w$ is a vertex of $c$ and is a minimum of $f\chi_c\}\subset {\rm Lk}(v, X).$

The {\em descending link} of a vertex is 

${\rm Lk}_{\downarrow}(v,X) = \bigcup \{{\rm Lk}(w,c)\mid c$ is a cube in $X$ and $\chi_c(w) = v$ and $w$ is a vertex of $c$ and is a maximum of $f\chi_c\}\subset {\rm Lk}(v, X).$
\end{definition}

\begin{thm} [Bestvina--Brady, \cite{bestvina_morse_1997}, Lemma 2.5]\label{homoeq}
Let $f$ be a Morse function. Suppose that $I\subset I'\subset\RR$ are connected and closed with $\min I = \min I'$ (resp. $\max I = \max I')$, and assume $I'\smallsetminus I$ contains only one point $r$ of $f\big(X^{(0)}\big)$. Then $X_{I'}$ is homotopy equivalent to the space obtained from $X_I$ by coning off the descending (resp. ascending) links of $v$ for each $v\in f^{-1}(r)$. 
\end{thm}

We can now deduce a lot about the topology of the level sets. We know how they change as we pass to larger intervals and so we have the following. 

\begin{cor}[Bestvina--Brady, \cite{bestvina_morse_1997}, Corollary 2.6]\label{cor1} Let $I,I'$ be as above.
\begin{enumerate}
\item If each ascending and descending link is homologically $(n-1)$-connected, then the inclusion $X_I\hookrightarrow X_{I'}$ induces an isomorphism on $H_i$ for $i\leq n-1$ and is surjective for $i=n$.
\item If the ascending and descending links are connected, then the inclusion $X_I\hookrightarrow X_{I'}$ induces a surjection on $\pi_1$. 
\item If the ascending and descending links are simply connected, then the inclusion $X_I\hookrightarrow X_{I'}$ induces an isomorphism on $\pi_1$.
\end{enumerate}
\end{cor}

Knowing that the direct limit of this system is a contractible space allows us to compute the finiteness properties of the kernel of $\phi$. 

\begin{theorem}[Bestvina--Brady, \cite{bestvina_morse_1997}, Theorem 4.1]\label{bbmorse}
Let $f\colon X\to \RR$ be a $\phi$-equivariant Morse function and let $H=\ker(\phi)$. If all ascending and descending links are simply connected, then $H$ is finitely presented (that is, $H$ is of type $F_2$).
\end{theorem}

We would also like to have conditions which will allow us to deduce that $H$ does not satisfy certain other finiteness properties. A well-known result in this direction is:

\begin{prop}[Brown, \cite{brown_cohomology_1982}, p. 193]\label{propb}
Let $H$ be a group acting freely, properly, cellularly and cocompactly on a cell complex $X$. Assume further that $\widetilde{H}_i(X,\ZZ)=0$ for $0\leq i\leq n-1$ and that $\widetilde{H}_n(X,\ZZ)$ is not finitely generated as a $\ZZ H$-module. Then $H$ is of type $FP_n$ but not $FP_{n+1}$.
\end{prop}

In \cite{brady_branched_1999}, the above result was used to prove that a certain group is not of type $FP_3$. In our theorems, not all the links will satisfy the assumptions of \cite[Theorem 4.7]{brady_branched_1999} and hence we require the following.

\begin{thm}\label{notfp}
Let $f\colon X\to \RR$ be a $\phi$-equivariant Morse function and let $H=\ker(\phi)$. Suppose that for all vertices $v$ the reduced homology of ${\rm Lk}_{\uparrow}(v)$ and ${\rm Lk}_{\downarrow}(v)$ vanishes in dimensions $0, \dots, n-1$ and $n+1$. Further assume that there is a vertex $v'$ such that $\widetilde{H}{_n}\big({\rm Lk}_{\uparrow}(v')\big)\neq 0$ or $\widetilde{H}{_n}\big({\rm Lk}_{\downarrow}(v')\big)\neq 0$ (possibly both). Then $H$ is of type $FP_n$ but not of type $FP_{n+1}$.
\end{thm}

\begin{proof}
From Corollary \ref{cor1} we know that $\widetilde{H}{_i}\big(X_{[t-N,t+N]}\big) = \widetilde{H}{_i}(X_t)$ for $0\leq i\leq n-1$ and all $N\in\NN$. Since homology commutes with direct limits, we can pass from $X_t$ to $X$ and deduce that all these homology groups are trivial. We now show that $\widetilde{H}{_n}{(X_t)}$ is not finitely generated as a $\ZZ H$-module. 

Suppose that we have a finite set of $n$-cycles $z_1, \dots, z_l$ that generate $\widetilde{H}{_n}(X_t)$ as a $\ZZ H$-module. There exists an $N$ such that each $z_i$ bounds an $(n+1)$-chain in $X_{[t-N,t+N]}$. So the inclusion induced map $\widetilde{H}{_n}{(X_t)}\to \widetilde{H}{_n}{(X_{[t-N,t+N]})}$ is zero. But by Corollary \ref{cor1} it is also onto, so $\widetilde{H}{_n(X_{[t-N,t+N]})} = 0$.

Theorem \ref{homoeq} implies that for every closed interval $J$ containing $[t-N,t+N]$, $X_J$ can be obtained (up to homotopy) from $X_{[t-N,t+N]}$ by coning off the ascending and descending links of vertices $v$ such that $f(v)\in\big(J\smallsetminus[t-N,t+N]\big)$. Since all these links have trivial homology in dimension $n+1$, we can see from the Mayer--Vietoris sequence that the inclusion induced map $$\widetilde{H}{_{n+1}}{(X_{J'})}\to\widetilde{H}{_{n+1}}{(X_{J})}$$ is injective for all $[t-N, t+N]\subseteq J'\subseteq J$. 

We will assume that the vertex $v'$ satisfies $\widetilde{H}{_n}\big({\rm Lk}_{\uparrow}(v')\big)\neq 0$. The case with descending links is analogous. Using the fact that the short exact sequence $$0\to H\to G\to \ZZ\to0$$ splits, we get a $\ZZ$ action on $X$ which gives us an action on the collection of level sets with $r\in\ZZ$ sending $X_t$ to $X_{t+r}$. Translating by this action, we can assume that the vertex $v'$ is in $X_{[t-N-s,t+N]}\smallsetminus X_{[t-N,t+N]}$ for some large $s$. Let $L$ be the union of all the ascending links coned off in the process of going from $X_{[t-N,t+N]}$ to $X_{[t-N-s,t+N]}$.

Once again, looking at the Mayer--Vietoris sequence we see that $$\widetilde{H}{_{n+1}}{\big(X_{[t-N-s,t+N]}\big)\to\widetilde{H}{_{n}}{(L)}}$$ is a surjective map. Since the latter is non-zero by assumption, the former must also be non-zero. This implies that $\widetilde{H}{_{n+1}}(X)$ is non-zero since homology commutes with direct limits; but $X$ is a contractible space, so we have a contradiction.

Therefore $\widetilde{H}_n(X_t)$ must be infinitely generated, so by Proposition \ref{propb} the result follows. 
\end{proof}

\section{Hyperbolisation of products of graphs}\label{bradygen2}

In this section we are going to take branched covers of products of graphs and prove that the result has hyperbolic fundamental group, with a finitely presented subgroup that is not of type $F_3$. We will do this in the following way. 

\begin{enumerate}
\item Start with three 2-full graphs $\Gamma_1$, $\Gamma_2$ and $\Gamma_3$ (Definition \ref{2full}).
\item Find a locally isometric copy of $L = \Gamma_1\sqcup\Gamma_2\sqcup\Gamma_3$ in $K = \Gamma_1\times\Gamma_2\times\Gamma_3$. 
\item Take a cover of $K\smallsetminus L$ and complete to get a branched cover $X$ of $K$. 
\item Define a Morse function on the universal cover $\tilde{X}$ of $X$ and examine the ascending and descending links of this function; from this we will conclude that there is a finitely presented subgroup that is not of type $F_3$. 
\item Finally, prove that $\tilde{X}$ is a hyperbolic space. 
\end{enumerate}

\begin{definition}\label{2full}
A connected graph $\Gamma$ is {\em 2-full} if the vertices of $\Gamma$ can be divided into 2 sets $A$ and $B$ such that every edge has one endpoint in $A$ and the other in $B$. 
\end{definition}

For simplicial graphs, this is equivalent to being bipartite. However, we allow multiple edges in our graphs and will reserve the use of ``bipartite'' for simplicial graphs. 

We will prove the following. 

\begin{customthm}{A}\label{bradygen}
Let $\Gamma_1$, $\Gamma_2$ and $\Gamma_3$ be three 2-full graphs such that the valence of any vertex is at least $4$. Then there is a finite branched cover $X$ of $\Gamma_1\times\Gamma_2\times\Gamma_3$ such that there are no isometrically embedded flat planes in $\tilde{X}$. Furthermore, there is a finitely presented subgroup of $\pi_1(X)$ that is not of type $F_3$.
\end{customthm}

\subsection{The branched cover}

Let the vertices of $\Gamma_i$ be divided into two sets $A_i, B_i$ as in the definition of 2-full and let $K = \Gamma_1\times\Gamma_2\times\Gamma_3$. 

Our branching locus will be 
$$L = (\Gamma_1  \times A_2      \times B_3)\sqcup
	  (B_1       \times\Gamma_2  \times A_3)\sqcup
	  (A_1       \times B_2      \times\Gamma_3).$$

This is a locally convex subcomplex of $K$. We now check that $\lk(c, K)\smallsetminus L$ is connected and non-empty for all cubes in $L$. We will do the case of a cube on $\Gamma_1\times A_2\times B_3$; the other cases are identical. If $c$ is the edge $[0,1]\times \{(a, b)\}$, then $\lk(c, K) = S^0\ast\lk(a, \Gamma_2)\ast\lk(b, \Gamma_3)$. We are removing the copy of $S^0$ and thus obtain a deformation retract onto $\lk(a, \Gamma_2)\ast\lk(b, \Gamma_3)$. If $c$ is the vertex $(v_1, a, b)$, then $\lk(c, K) = \lk(v_1, \Gamma_1)\ast\lk(a, \Gamma_2)\ast\lk(b, \Gamma_3)$; we are removing the copy of $\lk(v_1, \Gamma_1)$. The complex $\lk(v_1, \Gamma_1)\ast\lk(a, \Gamma_2)\ast\lk(b, \Gamma_3)\smallsetminus \lk(v_1, \Gamma_1)$ deformation retracts onto a join of two non-empty sets. Therefore, $\lk(c, K)\smallsetminus L$ is connected and non-empty for all cubes in $L$.

We now consider the complex $K\smallsetminus L$. We have three surjective projections:
\begin{align*}
p_1\colon K\smallsetminus L&\to \Gamma_2\times\Gamma_3\smallsetminus (A_2\times B_3),\\
p_2\colon K\smallsetminus L&\to \Gamma_3\times\Gamma_1\smallsetminus (A_3\times B_1),\\
p_3\colon K\smallsetminus L&\to \Gamma_1\times\Gamma_2\smallsetminus (A_1\times B_2),
\end{align*}
which are the restrictions of the projection maps $K\to \Gamma_i\times\Gamma_j$.

\begin{prop}
Each of the complexes $\Gamma_i\times\Gamma_j\smallsetminus(A_i\times B_j)$ deformation retracts onto a graph. 
\end{prop}
\begin{proof}
We start at a removed vertex and push out radially in adjacent squares. This removes the interior of each cube adjacent to a vertex in $A_i\times B_j$. As every edge in a 2-full graph has one endpoint in $A$ and one in $B$, all 2-cells are retracted by this process, leaving us with a graph $\mathcal{G}_{ij} := (\Gamma_i\times A_j)\cup (B_i\times\Gamma_j)\subset \Gamma_i\times\Gamma_j$.
\end{proof}

Let $q_{ij}>\max\{\mbox{deg}(v):v\in\Gamma_i\sqcup\Gamma_j\}$ be a prime number and $S_{q_{ij}}$ be the symmetric group on $q_{ij}$ elements. Let $\alpha\in S_{q_{ij}}$ be a $q_{ij}$-cycle and $\beta\in S_{q_{ij}}$ be an element such that $\beta\alpha\beta^{-1} = \alpha^l$, where $l$ is a generator of $\ZZ_{q_{ij}}^{\times}$.

A homomorphism $\pi_1(\mathcal{G}_{ij})\to S_{q_{ij}}$ can be defined by labelling each edge of $\mathcal{G}_{ij}$ with an element of $S_{q_{ij}}$, with the condition that if $e$ has label $g$, then $\bar{e}$ has label $g^{-1}$.

We start by labelling the edges of $\Gamma_i$ and $\Gamma_j$, then extend this to a labelling of $\mathcal{G}_{{ij}}$ in an obvious way. 

For each edge of $\Gamma_i$ we label with an element of $\{\alpha, \dots, \alpha^{q_{ij}}\}$, such that at each vertex the edges oriented towards it are labelled by different powers of $\alpha$. We similarly label the edges of $\Gamma_j$ with powers of $\beta$ such that at each vertex the edges oriented towards it are labelled by different powers of $\beta$. 

This defines a homomorphism $\rho_{ij}\colon\pi_1(\Gamma_i\times\Gamma_j\smallsetminus(A_i\times B_j))\to S_{q_{ij}}$. A loop of length 4 in the link of a removed vertex comes from a diagram as in Figure \ref{loopsoflength}, where the centre vertex is removed. This loop deformation retracts onto a loop of length 8 in $\mathcal{G}_{ij}$. This loop is labelled by a commutator of the form $[\beta^{-m}, \alpha^{-n}] = \beta^m\alpha^n\beta^{-m}\alpha^{-n} = \alpha^{n(l^m -1)}$. When $n<q_{ij}$ and $m<q_{ij}-1$, this is a non-trivial power of $\alpha$. These conditions will be satisfied by the choice of $q_{ij}$, hence this commutator is a $q_{ij}$-cycle. 

\begin{figure}\center
\def\svgwidth{200pt}
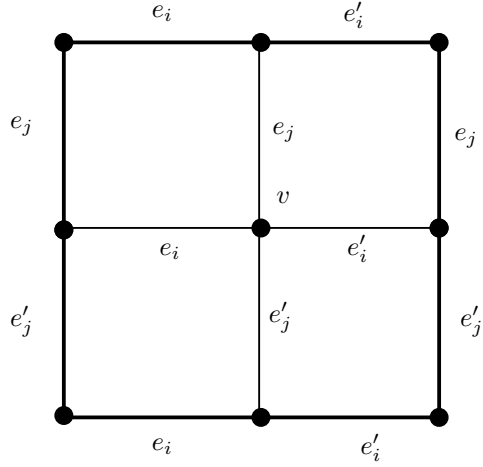
\caption[Loops of length 4.]{The loops $e_i$ and $e_i'$ are in $\Gamma_i$ and $e_j$ and $e_j'$ are in $\Gamma_j$. The vertex labelled $v$ is in $A_i\times B_j$.}
\label{loopsoflength}

\end{figure}

We now have three maps:
\begin{align*}
\rho_{23}\circ (p_1)_*\colon \pi_1(K\smallsetminus L)&\to S_{q_{23}},\\
\rho_{31}\circ (p_2)_*\colon \pi_1(K\smallsetminus L)&\to S_{q_{31}},\\
\rho_{12}\circ (p_3)_*\colon \pi_1(K\smallsetminus L)&\to S_{q_{12}}.
\end{align*}
Let $q = q_{12}q_{23}q_{31}$. We combine the above representations to get a representation $\rho\colon \pi_1(K\smallsetminus L)\to S_{q}$. If each representation as seen as acting on the vector space $\RR^{q_{ij}}$, then $\rho$ should be the tensor product of these representations. We then take the cover of $K\smallsetminus L$ corresponding to the stabiliser of $1$ in $S_q$ and complete to get the desired branched cover $X$. By Lemma \ref{npcbranch}, this cover will be non-positively curved.

We now consider links of vertices in $X$ and construct a Morse function.

The vertices of $X$ can be split into 2 types:
\begin{enumerate}
\item Vertices that do not map to $L$.
\item Vertices that map to $L$. 
\end{enumerate}

Given a vertex $v$ of type 1, let $w = (v_1, v_2, v_3)$ be the vertex in $K$ to which $v$ maps. Since $w$ is disjoint from the branching locus, a small neighbourhood lifts to $X$. The link of $v$ is therefore isomorphic to the link of $w$, that is, $$\lk(v, X) = \lk(v_1, \Gamma_1)\ast \lk(v_2, \Gamma_2)\ast \lk(v_3, \Gamma_3).$$

We now examine the link of a vertex $v$ of type 2. Let $w = (v_1, v_2, v_3)$ be the vertex in $K$ to which $v$ maps. We will do the case where $v_2\in A_2$ and $v_3\in B_3$; the other cases can be treated similarly.

Recall that $\lk(w, K) = \lk(v_1, \Gamma_1)\ast \lk(v_2, \Gamma_2)\ast \lk(v_3, \Gamma_3)$. Further recall that during the branching process we are removing the set corresponding to the vertices in $\lk(v_1, \Gamma_1)$. We have a map on fundamental groups $$\pi_1(\lk(w, K)\smallsetminus \lk(v_1, \Gamma_1))\to \pi_1(K\smallsetminus L).$$ We consider the image of $\pi_1(\lk(w, K)\smallsetminus \lk(v_1, \Gamma_1))$ under the maps ${p_{i}}_*$. 

$\lk(w, K)\smallsetminus \lk(v_1, \Gamma_1)$ is mapped by $p_1$ to the link of a removed vertex in 
$$\Gamma_2\times\Gamma_3\smallsetminus(A_2\times B_3);$$ 
the image is thus isomorphic to $\lk(v_2, \Gamma_2)\ast \lk(v_3, \Gamma_3)$. The space $(A\ast B)\smallsetminus B$ is homotopy equivalent to $A$; this homotopy equivalence comes from the projection 
$$(C_0(A)\times C_0(B))\smallsetminus (\{0\}\times C_0(B))\to C_0(A),$$ 
where $C_0(X)$ is the CAT(0) cone from Definition \ref{conedef}. The map $\lk(w, K)\smallsetminus \lk(v_1, \Gamma_1)\to \lk(v_2, \Gamma_2)\ast \lk(v_3, \Gamma_3)$ coming from the projection is a homotopy equivalence. Under the maps $p_2$ and $p_3$, $\lk(w, K)\smallsetminus \lk(v_1, \Gamma_1)$ is sent to a union of contractible sets. We are now reduced to considering the image of $\pi_1(\lk(w, K)\smallsetminus \lk(v_1, \Gamma_1))$ under the map $\rho_1$ to $S_{q_{23}}$. We picked the maps such that loops of length 4 in $\lk(v_2, \Gamma_2)\ast \lk(v_3, \Gamma_3)$ are sent to $q_{23}$-cycles under the homomorphism $\rho_1$.

We can cover the graph $\lk(v_2, \Gamma_2)\ast \lk(v_3, \Gamma_3)$ with a sequence $\mathcal{L}$ of loops of length 4, such that each loop has non-empty intersection with the union of the previous loops. Every loop of length 4 has connected preimage in the cover. The cover of each loop in $\mathcal{L}$ will have non-empty intersection with the union of previous loops. Thus the corresponding cover will be a connected graph with no loops of length 4, as any loop of length 4 has preimage a loop of length $>4$. We will denote this graph $\Lambda_v$. The link of $v$ is the join of $\Lambda_v$ with the discrete set $\lk(v_1, \Gamma_1)$.

\subsection{The Morse function and ascending and descending links}

We put an orientation on each edge of $\Gamma_i$; we orient the edges such that at each vertex there are at least two incoming edges and at least two outgoing edges. We give the circle $S^1$ a cell structure with 1 vertex and 1 edge, putting an orientation on this edge. Define a map $f_i\colon \Gamma_i\to S^1$ by mapping each edge to the edge of $S^1$ by the given orientation and mapping all vertices to the vertex of $S^1$. These assignments define a map
\begin{align*}
f\colon  \Gamma_1\times\Gamma_2\times\Gamma_3&\to S^1 = \RR/\ZZ,\\
f(x,y,z)& = f_1(x)+f_2(y)+f_3(z),
\end{align*}
which lifts to an $f_*$-equivariant Morse function $\tilde{f}\colon \tilde{K}\to\RR$.

We precompose with the branched covering map $b$ to get a map $h = f\circ b\colon X\to S^1$. Lift this to universal covers to get a Morse function $\tilde{h}\colon \tilde{X}\to\RR$, which is $h_*$-equivariant. We will examine the ascending and descending links of $\tilde{h}$. We concentrate on the case of the ascending link; the arguments are the same for the descending link. 

Note that $\tilde{h} = \tilde{f}\circ\tilde{b}$ where $\tilde{b}\colon \tilde{X}\to\tilde{K}$ is the lift of $b$ to universal covers. We also consider the induced map on links $\tilde{b}_{\lk(v)}\colon \lk(v, \tilde{X})\to \lk(\tilde{b}(v), \tilde{K})$. Given a vertex $v$ of $\tilde{X}$ the ascending link is the preimage of the ascending link of $\tilde{b}(v)$. The ascending link of $\tilde{b}(v)$ is $U_1\ast U_2\ast U_3$ where $U_i$ is a discrete set whose size is the number of edges oriented away from $\tilde{b}(v)$ in $\tilde{\Gamma}_i$.

If $v$ is a type 1 vertex, then the ascending link of $v$ is isomorphic to the ascending link of $\tilde{b}(v)$, which is $U_1\ast U_2\ast U_3$. This join is simply connected but has non-zero second homology; indeed, as each $|U_i|>1$ it has the homotopy type of a wedge of $2$-spheres.

Let $v$ be a vertex of type 2; without loss of generality, assume that it lies on a lift of $\Gamma_1\times A_2\times B_3$. The ascending link of $v$ will be the preimage of $U_1\ast U_2\ast U_3$ in $U_1\ast\Lambda_v$. Since $U_i$ contains at least 2 points,  the graph $U_2\ast U_3$ is connected but not simply connected. We cover $U_2\ast U_3$ with loops of length 4 such that successive loops have non-empty intersection with the union of the previous loops. Since each of these loops has connected preimage in $\Lambda_v$, the intersection of a loop with the union of the previous loops will remain non-empty. Therefore, the cover $\mathcal{U}$ of $U_2\ast U_3$ in $\Lambda_v$ will be connected. The ascending link will be $\mathcal{U}\ast U_1$ which will be simply connected as the join of a connected set and a non-empty set. 

The kernel of $f_*$ is finitely presented by Theorem \ref{bbmorse}. The ascending link of a vertex which is of type 1 has the form $U_1\ast U_2\ast U_3$, which although simply connected has non-zero second homology. Applying Theorem \ref{notfp}, we see that the kernel will not be of type $FP_3$.

\subsection{Hyperbolicity of the branched cover}

To complete the proof of Theorem \ref{bradygen} we must prove that $\tilde{X}$ is hyperbolic.

Let $\tilde{L}$ be the preimage of $L$ in $\tilde{X}$. We will prove that the cube complex $\tilde{X}$ is hyperbolic by contradiction. By Theorem \ref{flatplane} it is enough to show that there are no isometrically embedded copies of $\EE^2$ in $\tilde{X}$. Suppose that $i\colon \EE^2\to\tilde{X}$ is such an embedding and let $F = i(\EE^2)$. First, we show that $i(\EE^2)$ has a transverse intersection point with $\tilde{L}$. Around such an intersection point we will see that the angle sum in $i(\EE^2)$ is forced to be $>2\pi$, contradicting the assumption that $i$ is an isometric embedding. 

\begin{figure}\center
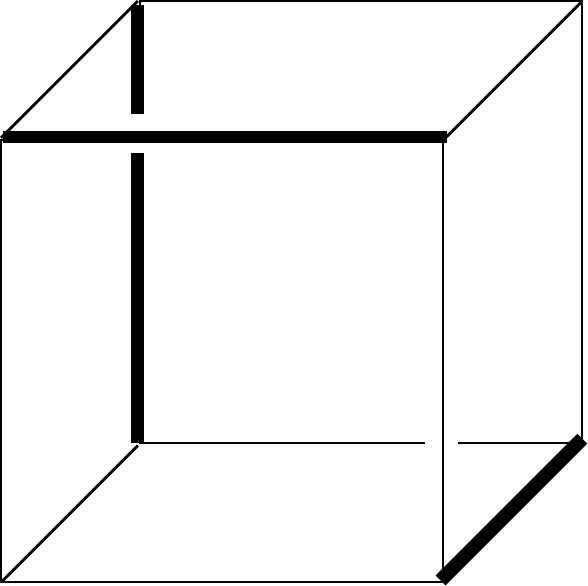
\caption{Intersection pattern of $\tilde{L}$ on a cube in $\tilde{X}$.}
\label{intpattern}
\end{figure}

Each 3-cube of $\tilde{X}$ intersects $\tilde{L}$ in 3 edges in the pattern depicted in Figure \ref{intpattern}. 

\begin{figure}\center
\def\svgwidth{200pt}
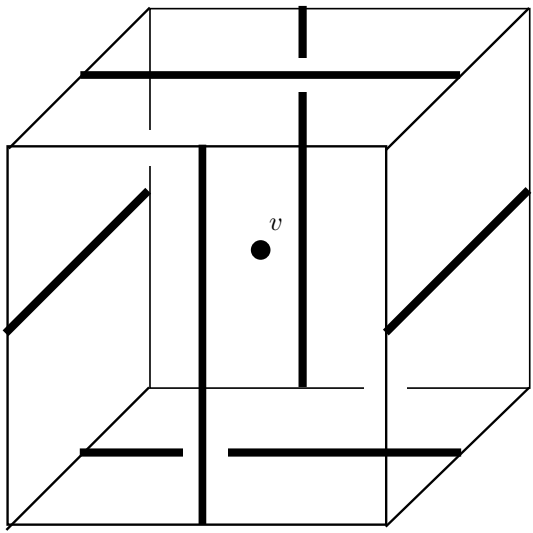
\caption{Cubes around a vertex not mapping to $L$.}
\label{type1vertex}
\end{figure}

Because $i$ is an isometric embedding, it cannot be contained in the 1-skeleton of $\tilde{X}$, so there is a 3-cube $c$ where the intersection is 2-dimensional. 

There are two cases to consider:
\begin{enumerate}
\item $F\cap\tilde{L}\cap c = \emptyset$
\item $F\cap\tilde{L}\cap c \neq \emptyset$
\end{enumerate}

In case 1, $F$ must have intersected $c$ in an edge adjacent to a vertex $v$ of $c$ which does not map to $L$. We can develop $\EE^2$ in a cubical neighbourhood of this vertex. The link of $v$ in this neighbourhood will produce part of an octahedron. We have depicted this in Figure \ref{type1vertex}. We can consider this cube as a subset of $\EE^3$, where $v$ sits at the origin and the edge with non-empty intersection is on one of the coordinate axes. Developing the plane into these cubes we see that we will intersect $L$ in one of the adjacent cubes from Figure \ref{type1vertex}.

We have now reduced to case 2. We have shown that there is a cube $c$ and an edge $e\subset\tilde{L}\cap c$ such that $F\cap e\neq \emptyset$. Either this intersection is a single point in which case we will see that it is a transverse intersection. If this is not the case, then the edge $e$ is entirely contained in $F$. In this case we can see from the intersection pattern in Figure \ref{intpattern} that $i(\EE^2)$ intersects one of the other edges of $\tilde{L}$ in $c$ in a single point. 

If the intersection point $x$ is in the interior of an edge, then it is a transverse intersection point. This can be checked by using a ball whose radius is less than the minimum distance to an endpoint of $e$.

In the case where the intersection point $x$ is the endpoint of an edge, we will prove that it is a transverse point of intersection by contradiction. 

If $x$ is not a transverse point of intersection, then there are points of $\EE^2\cap\tilde{L}$ arbitrarily close to $x$. Since $i$ is an isometric embedding, if two points of an edge are in $i(\EE^2)$, then the whole edge is contained in $i(\EE^2)$. Therefore, $i(\EE^2)$ will contain both a 2-dimensional cross section of the cube $c$ and an edge $e'$ of $\tilde{L}$ in another cube that is adjacent to $c$. Since the link of $x$ is the join of a graph and a discrete set and the vertex corresponding to $e'$ is contained in the discrete set, it follows that there is a cube $c'$ containing $e'$ such that $c$ and $c'$ share a 2-dimensional face. Consider the union of these two cubes as a subset of $\EE^3$. The plane defined by the polygon in $c$ does not contain the edge in $\tilde{L}$, so $i$ cannot have been an isometric embedding. 

We now have a transverse intersection point with $\tilde{L}$. To see that the flat plane is not isometrically embedded we split into 2 cases. 

Firstly, we will study the case where the transverse intersection point $x$ is in the interior of an edge $e$. The plane will intersect several cubes in various polygonal subsets each contributing an angle. Since we assumed that the plane was isometrically embedded, these angles must sum to $2\pi$. Two such polygons joined along an edge correspond to two cubes meeting along a face containing $e$. These polygons make a contribution of $\pi$ to get an angle sum of $2\pi$ we would therefore have to have exactly four cubes meeting around $e$. But the link of $e$ in $\tilde{X}$ contains no loops of length less than 6. 

If the transverse intersection point $x$ is the endpoint of the edge $e$ we proceed similarly. We will assume that the component of $\tilde{L}$ with the transverse intersection is a lift of $\Gamma_1$. Since there is a flat containing $x$ we see that there is a loop of length $2\pi$ in $\lk(x, \tilde{X})$. We can now take the CAT(0) cone $C_0(x)$ over $\lk(x, \tilde{X})$. Let $\Lambda_x$ be the subgraph of $\lk(x, \tilde{X})$ spanned by vertices not corresponding to $\tilde{L}$. Since $\lk(x, \tilde{X}) = \Lambda_x\ast \lk(x, \tilde{L})$ we see by Theorem \ref{joinisproduct} $C_0(x)$ is isometric to $C_0(\Lambda_x)\times C_0(\lk(x, \tilde{L}))$. The latter summand is several copies of $\RR_+$ joined at $0$. We can project this plane to $C_0(\Lambda_x)$. Two parallel geodesics have bounded distance under this projection and thus by Theorem \ref{flatstrip} they bound a flat strip. This gives a loop of length $2\pi$ in $\Lambda_x$. However, by choice of the branched covering, $\Lambda_x$ has no loops of length less than 6.

Our argument by contradiction is complete, and we have proved that there cannot be an isometrically embedded flat plane. Thus $\tilde{X}$ is a hyperbolic CAT(0) cube complex.

\subsection{Euler characteristics}

Choosing different graphs $\Gamma_i$ we have an infinite family of spaces, each of which has hyperbolic fundamental group. We will now calculate the Euler characteristic of these groups to show that this family contains an infinite subfamily of pairwise non-isomorphic hyperbolic groups. We would like to improve this result to show that the non-hyperbolic subgroups of interest are not commensurable; this work is ongoing. 

\begin{prop}\label{euler1}
Let $\Gamma_1$, $\Gamma_2$ and $\Gamma_3$ be 2-full graphs, let $V_i = A_i\sqcup B_i$ be the vertex set of $\Gamma_i$ and let $E_i$ be the edge set of $\Gamma_i$. We use the shorthand $|A_i| = a_i, |B_i| = b_i, |V_i| = v_i$ and $|E_i| = e_i$. Then the $q_{12}q_{23}q_{31}$-fold branched cover $X$ of $K = \Gamma_1\times\Gamma_2\times\Gamma_3$ has Euler characteristic
\begin{multline*}
q_{12}q_{23}q_{31}(v_1v_2v_3 + a_1b_2e_3 + b_1e_2a_3 + e_1a_2b_3 + e_1e_2v_3 + e_1v_2e_3 + v_1e_2e_3\\
 - v_1a_2b_3 - b_1v_2a_3 - a_1b_2v_3 - e_1v_1v_2 - v_1e_2v_3 - v_1v_2e_3 - e_1e_2e_3)\\
- q_{12}q_{31}a_2b_3(e_1 - v_1)
- q_{12}q_{23}a_3b_1(e_2 - v_2)
- q_{23}q_{31}a_1b_2(e_3 - v_3)
\end{multline*}

\end{prop}
\begin{proof}
To prove this we count the number of cells of each dimension in $X$. We do this by examining the cells in the cover of $K\smallsetminus L$, accounting for how many cells we add when we complete to obtain $X$. 

There are no 2 or 3-cells in $L$, therefore each such cell in $K$ lifts to the cover. Thus we have $$q_{12}q_{23}q_{31}(e_1e_2v_3 + e_1v_2e_3 + v_1e_2e_3)$$ 2-cells and $$q_{12}q_{23}q_{31}(e_1e_2e_3)$$ 3-cells in $X$. 

We are removing the $$a_1b_2e_3 + b_1e_2a_3 + e_1a_2b_3$$ 1-cells of $L$ from $K$, so in the cover of $K\smallsetminus L$ we have $$q_{12}q_{23}q_{31}(e_1v_1v_2 + v_1e_2v_3 + v_1v_2e_3 - a_1b_2e_3 - b_1e_2a_3 - e_1a_2b_3)$$ 1-cells before completing. 

Similarly, we are removing the $$v_1a_2b_3 + b_1v_2a_3 + a_1b_2v_3$$ 0-cells of $L$ from $K$, so in the cover of $K\smallsetminus L$ we have $$q_{12}q_{23}q_{31}(v_1v_2v_3 - v_1a_2b_3 - b_1v_2a_3 - a_1b_2v_3)$$ 0-cells. 

We now account for the cells added on completing the cover. We  look at the completion of a component of $\Gamma_1$. The link of a lift of an edge $e$ in $\Gamma_1$ will be a $q_{23}$ fold cover of the link of $e$. The lift of $\Gamma_1$ will thus be a $q_{12}q_{31}$-fold cover of $\Gamma_1$. This cover will have $q_{12}q_{31}(v_1a_2b_3)$ 0-cells and $q_{12}q_{31}(e_1a_2b_3)$ 1-cells. Repeating this count for the graphs $\Gamma_2$ and $\Gamma_3$ gives the desired result. 
\end{proof}

A simple example of a 2-full graph $\Gamma$ satisfying the hypothesis of Theorem \ref{bradygen} is obtained by taking a prime $p>3$ and $p-1$ copies of $[0,1]$, then identifying all the 0 endpoint and all the 1 endpoints. This is a ``cage graph'' with 2 vertices and $p-1$ edges. The complex $X$ constructed as above by taking $\Gamma_1 = \Gamma_2 = \Gamma_3 = \Gamma$ and $q_{12} = q_{13} = q_{23} = p$ has Euler characteristic $p^2(9+15p - 24p^2+9p^3-p^4)$. Thus we obtain an infinite family of hyperbolic groups with finitely presented subgroups not of type $F_3$.

\section{Another family of groups which are finitely presented and not of type \texorpdfstring{$F_3$}{F3}}

The second method we will present extends Lodha's work \cite{lodha_finiteness_2017}.

\begin{definition}\label{sizeabledef}
A simplicial graph $\Gamma$ is {\em sizeable} if it satisfies the following conditions:
\begin{itemize}
\item $\Gamma$ is bipartite on two sets $A$ and $B$,
\item $\Gamma$ contains no loops of length 4, 
\item there exist partitions $A = A^+\sqcup A^-$ and $B = B^+\sqcup B^-$, such that $\Gamma(A^s\sqcup B^t)$ is connected for all $s,t\in\{-,+\}$.
\end{itemize}
\end{definition}

A sizeable graph with 44 vertices was constructed in \cite{lodha_finiteness_2017}. One can construct many examples using the procedure for hyperbolising 2-dimensional right angled Artin groups detailed in \cite{kropholler_almost_2017}. In \cite{kropholler_thesis}, an example with 37 vertices is given. It is now natural to ask for the minimal number of vertices and edges in a sizeable graph. Theorem \ref{thm:sizeable24} of the appendix shows that the minimal number of vertices of a sizeable graph is 24. Furthermore such a graph is constructed. 

The following remark differentiates our complexes from those constructed in \cite{lodha_finiteness_2017, brady_branched_1999}.

\begin{rem}\label{difftimes}
	The graph depicted in Figure \ref{fig:sizeable24} has all four subgraphs contractible. When we create the cube complex $X$ from Theorem \ref{lodhagen}, the ascending and descending links come in two types: some are contractible, the others are the join of three discrete sets each with 2 points.
\end{rem}

The example of Brady \cite{brady_branched_1999} can be seen as the special case of Theorem \ref{bradygen} based on a cage graph with 4 edges, where in the proof one takes $q_{12} = q_{23} = q_{31} = 5$. The example of Lodha \cite{lodha_finiteness_2017} is a special case of Theorem \ref{lodhagen}, using the sizeable graph defined in \cite{lodha_finiteness_2017}. In both cases the ascending and descending links come in two varieties: each is either a join of 3 discrete sets, each with 2 points, or the suspension of a loop of length at least 20. Thus we can see that the ascending and descending links in the construction can be quite different from the previous constructions. 

We now move on to proving the main theorem of this section.

\begin{customthm}{B}\label{lodhagen}
For $i=1,2$ and $3$, let $\Gamma_i$ be a sizeable graph with vertex set $A_i\sqcup B_i$. Let $K_{ij}$ be the complete bipartite graph on $A_i$ and $B_j$. Let $X$ be the full cubical subcomplex of $K_{13}\times K_{21}\times K_{32}$ spanned by  vertices $(v_1, v_2, v_3)\in K_{13}\times K_{21}\times K_{32}$ which satisfy one of the following conditions, 
\begin{itemize}
\item $v_i\in A_i$ for all $i$, 
\item $v_i\in B_{i-1}$ for all $i$, 
\item $v_1\in A_1, v_2\in B_1$ and $[v_1, v_2]$ is an edge of $\Gamma_1$, 
\item $v_2\in A_2, v_3\in B_2$ and $[v_2, v_3]$ is an edge of $\Gamma_2$, 
\item $v_3\in A_3, v_1\in B_3$ and $[v_3, v_1]$ is an edge of $\Gamma_3$.
\end{itemize}
Then $\pi_1(X)$ is hyperbolic and contains a finitely presented subgroup that is not hyperbolic.
\end{customthm}

We will give the proof of this theorem in several stages.
\begin{enumerate}
\item Terminology relating to the complex $X$. 
\item Define a Morse function $\tilde{X}\to\RR$, where $\tilde{X}$ is the universal cover of $X$. 
\item Examine the ascending and descending links of the Morse function to see that $\pi_1(X)$ has a finitely presented subgroup which is not of type $F_3$.
\item Prove that $\tilde{X}$ is hyperbolic. 
\end{enumerate}

\subsection{The complex \texorpdfstring{$X$}{X}}

The vertices of $X$ are defined in Theorem \ref{lodhagen}. We say that a vertex is of type 1 if it satisfies either of the first two conditions and of type 2 otherwise. We include a cube in $X$ if all the vertices defining it are in $X$.

Put an orientation on each edge of $K_{ij}$ by orienting it from $A_i^s$ to $B_j^t$ if $s=t$ and orienting it towards $A^s_i$ otherwise. Give $S^1$ a cell structure with one vertex and one oriented edge. We define a map $h_i\colon K_{ij}\to S^1$ by mapping open edges homeomorphically to the open edge of $S^1$ respecting orientation, and we extend this to a map 
\begin{align*}
h&: K_{13}\times K_{21}\times K_{32}\to S^1,\\
h(x,y,z)&:=  h_1(x)+h_2(y)+h_3(z).
\end{align*}
Restricting to $X$, we get a map $f\colon X\to S^1$.  Lifting to universal covers, we get a Morse function $\tilde{f}\colon \tilde{X}\to \RR$ which is $f_*$-equivariant.

\subsection{The ascending and descending links of \texorpdfstring{$f$}{f}}

We first examine the links of vertices and prove that $X$ is a non-positively curved cube complex. We will then move on to looking at the ascending and descending links. 

\begin{notation}
For a vertex $v$ in $\Gamma_i$, let $N_{v}$ be the set of vertices adjacent to $v$ in $\Gamma_i$. 
\end{notation}

\subsubsection{A type 1 vertex}

Let $v = (v_1,v_2,v_3)$ be a vertex of type 1. We will consider the case where $v_i\in A_i$, the other case (where $v_i\in B_{i-1}$) being similar. Consider adjacent vertices in $K_{13}\times K_{21}\times K_{32}$ that are of the form $(v_1',v_2,v_3), (v_1,v_2',v_3)$ and $(v_1,v_2,v_3')$ where $v_i'\in B_{i-1}$. These vertices are in $X$ under the following conditions:

\begin{itemize}
\item $(v_1',v_2,v_3)$ is in the complex if $[v_1', v_3]$ is an edge of $\Gamma_3$. 
\item $(v_1,v_2',v_3)$ is in the complex if $[v_2', v_1]$ is an edge of $\Gamma_1$. 
\item $(v_1,v_2,v_3')$ is in the complex if $[v_3', v_2]$ is an edge of $\Gamma_2$. 
\end{itemize}
Thus $Lk(v, X)^{(0)} = N_{v_1}\sqcup N_{v_2}\sqcup N_{v_3}$.

We now look at which edges will be in $Lk(v, X)$. There is an edge between the vertices corresponding to $(v_1,v_2,v_3')$ and $(v_1,v_2',v_3)$ if the vertex $(v_1,v_2',v_3')$ is in the complex, since this will mean that the edges $[(v_1,v_2,v_3), (v_1,v_2',v_3)]$ and $[(v_1,v_2,v_3), (v_1,v_2,v_3')]$ are adjacent on a square. The vertex $(v_1,v_2',v_3')$ is in the complex if $[v_1, v_2']$ is an edge of $\Gamma_1$ which is the case above. Thus $Lk(v, X)^{(1)} = (N_{v_1}\ast N_{v_2}\ast N_{v_3})^{(1)}$. To understand the 2-skeleton we check which cubes are in $X$. This corresponds to checking vertices of the form $(v_1',v_2',v_3')$: in fact $v_i'\in B_{i-1}$ always, so this vertex is always in the complex. We conclude that $Lk(v, X) = N_{v_1}\ast N_{v_2}\ast N_{v_3}$.

We now examine the ascending and descending links of $v$. We will examine the case of the ascending link, the descending link being similar.

$Lk_{\uparrow}(v, X)$ is the full subcomplex of $Lk(v, X)$ corresponding to edges in the graphs $K_{ij}$ oriented away from $A_i$. This is the join of three sets $Q_{v_1},Q_{v_2}$ and $Q_{v_3}$ where $Q_{v_i}\subset N_{v_i}$ is the subset consisting of those edges in $K_{ij}$ oriented away from $A_i$. 

The ascending and descending links are all simply connected since they are the joins of three discrete sets. For a vertex $v_i\in \Gamma_i$, let $O_{v_i}\subset N_{v_i}$ be those edges oriented towards $v_i$. There is a vertex where at least one of $Q_{v_i}$ and $O_{v_i}$ contains at least 2 points. Indeed, if this were not the case then every vertex would have exactly one edge oriented towards it and one away, and since the subgraphs $\Gamma(A^s\sqcup B^t)$ are connected they would have to be segments meaning that $\Gamma$ is a copy of $S^0\ast S^0$ which contains a loop of length 4. Thus there is a vertex $v$ such that at least one of $Lk_{\uparrow}(v, X)$ or $Lk_{\downarrow}(v, X)$ has non-zero second homology. 

\subsubsection{A type 2 vertex} 

Now examine the link of a type 2 vertex. We will look at a vertex $v = (v_1,v_2,v_3)$, where $v_1\in A_1, v_2\in B_1, v_3\in A_3$ and $[v_1, v_2]$ is an edge of $\Gamma_1$. All other cases are similar.

We start by considering adjacent vertices, which give the 0-skeleton of $Lk(v, X)$. We can see that the possible adjacent vertices are of the form $(v_1',v_2,v_3), (v_1,v_2',v_3)$ or $(v_1,v_2,v_3')$. All vertices of the form $(v_1,v_2',v_3)$ are in $X$ as these are type 1 vertices. All vertices of the form $(v_1,v_2,v_3')$ are also in $X$ as $[v_1, v_2]$ is an edge of $\Gamma_1$. Vertices of the form $(v_1',v_2,v_3)$ are in the complex if $[v_1', v_3]$ is an edge of $\Gamma_3$. Therefore, $Lk(v, X)^{(0)} = A_2\sqcup B_2\sqcup N_{v_3}$.

We now consider which squares are in the complex $X$; these will enable us to compute $Lk(v, X)^{(1)}.$ We will do the three cases individually. 

\begin{enumerate}

\item There is an edge between vertices corresponding to $(v_1,v_2',v_3)$ and $(v_1,v_2,v_3')$ if the vertex $(v_1,v_2',v_3')$ is in the cube complex. Here, $v_1\in A_1, v_2'\in A_2, v_3'\in B_2$, which means this vertex is in the cube complex if $[v_2', v_3']$ is an edge of $\Gamma_2$. 

\item There is an edge between vertices corresponding to $(v_1',v_2,v_3)$ and $(v_1,v_2,v_3')$ if the vertex $(v_1',v_2,v_3')$ is in the cube complex. Here, $v_1'\in B_3, v_2\in B_1, v_3'\in B_2$, so this is a type 1 vertex and is always in the cube complex. 

\item Similarly, there is an edge between vertices corresponding to $(v_1',v_2,v_3)$ and $(v_1,v_2',v_3)$ if the vertex $(v_1',v_2',v_3)$ is in the complex. Now, $v_1'\in B_3, v_2'\in A_2, v_3\in A_3$. So the vertex $(v_1',v_2',v_3)$ is in the complex if $[v_1', v_3]$ is an edge of $\Gamma_3$, which is the case if $(v_1', v_2, v_3)$ is in the complex. 
\end{enumerate}

Putting all of this together, we can see that the 1-skeleton of $Lk(v, X)$ is the 1-skeleton of $\Gamma_2\ast N_{v_3}$.

To see the complex is flag we must consider the cubes in $X$. Three vertices $(v_1',v_2,v_3), (v_1,v_2',v_3), (v_1,v_2,v_3')$ which are pairwise adjacent in the link span a triangle if $(v_1',v_2',v_3')$ is in $X$. We can see that $v_1'\in B_3, v_2'\in A_2, v_3'\in B_2$ and this is in the complex as long as $[v_2', v_3']$ is an edge of $\Gamma_2$, which is always the case, as  $(v_1,v_2',v_3), (v_1,v_2,v_3')$ are adjacent in the link. This shows that $Lk(v, X) = \Gamma_2\ast N_{v_3}$. 

The ascending link of $v$ is the full subcomplex spanned by those edges which are oriented away from $v$ this is the join of $Q_{v_3}$ with the full subgraph of $\Gamma_2$ spanned by $A^s\cup B^t $ for some $s,t\in\{-, +\}$. This is simply connected since $\Gamma_2$ is sizeable. For the convenience of the reader we have depicted the link of each vertex in Figure \ref{ascdesclinks}.

\begin{figure}\center
	\centering
\def\svgwidth{350pt}
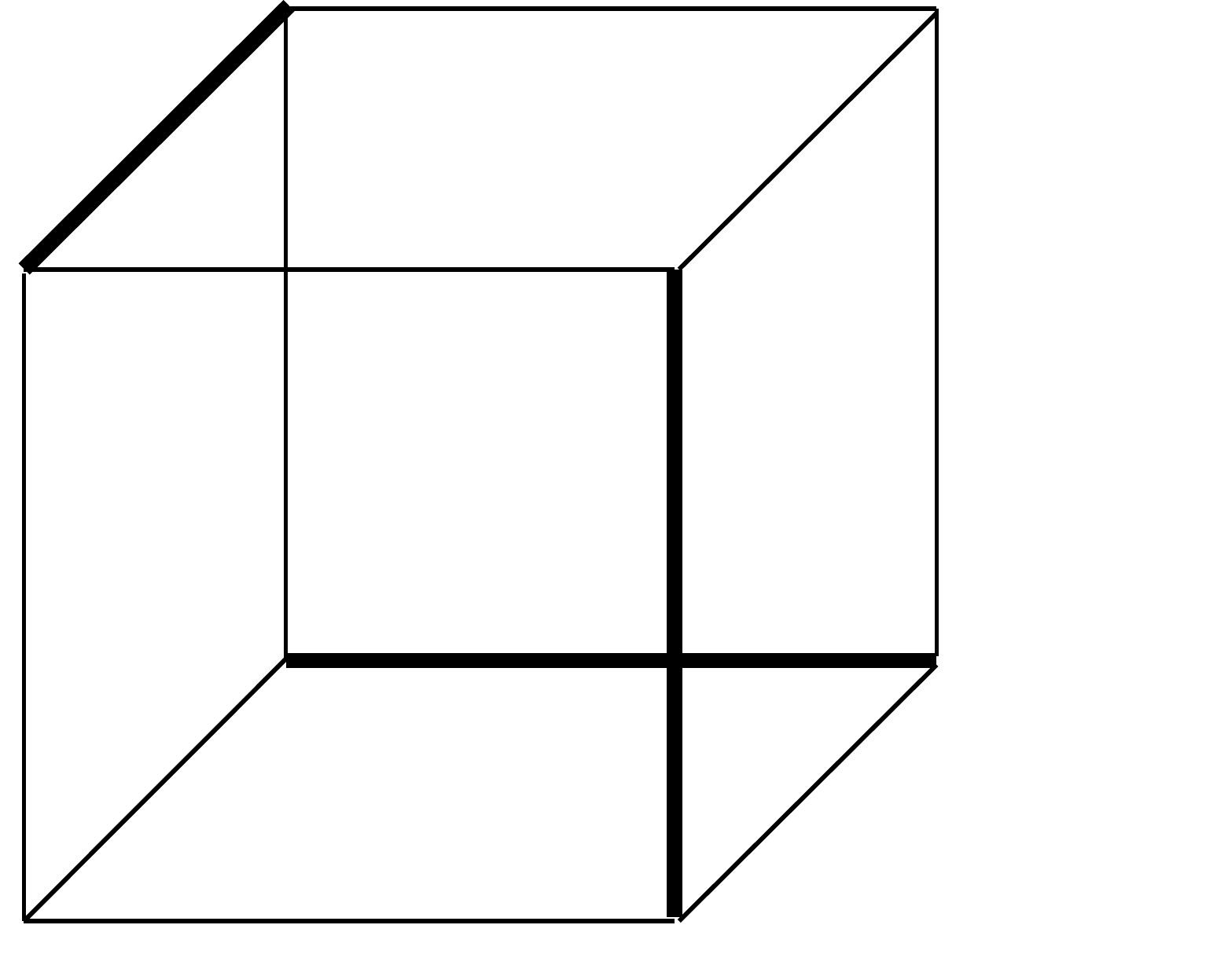
\caption[The link of each edge and vertex in $X$.]{The link of each vertex and each edge of $X$. Each $v_i\in A_i$ and $v_i'\in B_{i-1}$. The bold edges are $\Gamma$-edges.}
\label{ascdesclinks}
\end{figure}

\subsection{Proof that \texorpdfstring{$\tilde{X}$}{the universal cover of X} is hyperbolic}

To prove that $\tilde{X}$ is hyperbolic we use a similar argument to that of Section \ref{bradygen2}.

Recall that a subset $D$ of $X$ intersects $\EE^2$ {transversally} at a point $p$ if there is an $\epsilon>0$ such that $N_{\epsilon}(p)\cap D\cap \EE^2 = \{p\}$.

In each 3-cube there are six type 2 vertices and two type 1 vertices. Joining vertices of type 2, there are two types of edge, which can be determined by their link, which is either $\Gamma_i$ or a complete bipartite graph. This is shown in Figure \ref{ascdesclinks}. We will be concerned with those edges which have link $\Gamma_i$, which we will refer to as {\em $\Gamma$-edges}. Let $C$ be the union of all $\Gamma$-edges. In a 3-cube the $\Gamma$-edges have the same arrangement as the edges of $\tilde{L}$ from Figure \ref{intpattern}. We will prove that there cannot be an isometrically embedded flat plane in $\tilde{X}$, by arguing that any flat plane would have a transverse intersection with $C$ and that around such a transverse point the angle sum must be at least $3\pi$ -- a contradiction.

Assume that $i\colon \EE^2\to \tilde{X}$ is an isometric embedding. It is not contained in the 1-skeleton, so there is a cube $c$ whose intersection with $i(\EE^2)$ is 2-dimensional. If $i(\EE^2)\cap c\cap C = \emptyset$, then we have intersected $c$ in a neighbourhood of a type 1 vertex. We can develop this plane into cubes meeting at this type 1 vertex. If we develop the flat plane we see that it will intersect a $\Gamma$-edge in one of these cubes. This situation is shown in Figure \ref{type1vertex} where the bold edges are $\Gamma$-edges.

We have reduced to the case $i(\EE^2)\cap c\cap C \neq \emptyset$. From the intersection pattern pictured in Figure \ref{intpattern} the plane must intersect a $\Gamma$-edge in 1 point. 

If the intersection point is in the interior of an edge, then we can see that it is a transverse intersection by taking a ball of radius less than the distance to either endpoint $x$. 

If the intersection point with the $\Gamma$-edge is a vertex $v$, then we must prove that it intersects all $\Gamma$-edges at this vertex transversally. Since $i$ was an isometric embedding, if we do not intersect all $\Gamma$-edges at $v$ transversally, then there is a $\Gamma$-edge $e'$ entirely contained in $i(\EE^2)$. However, there is a cube $c'$ sharing a face with $c$ in which this edge $e'$ is contained. Thus $i(\EE^2)$ will contain a polygonal subset in $c$ and a $\Gamma$-edge $e'$ in $c'$. Considering $c\cup c'$ as a subset of $\EE^3$ we can see that the $\Gamma$-edge is not contained in the plane defined by the polygon in $c$. Therefore, the plane could not have been isometrically embedded. 

We conclude that there is a transverse intersection point with $C$. Examining the angle sum around such a point will give the desired result. 

Firstly, we will study the case where the intersection is in the interior of an edge $e$. In the link of such an intersection point we will see a loop of length $2\pi$. The plane will intersect several cubes in various polygonal subsets, each contributing an angle to this loop. These angles must sum to $2\pi$. A pair of such polygons joined along an edge corresponds to a pair of cubes which share a face meeting along $e$; the angle contribution from these two polygons is $\pi$. To get an angle sum of $2\pi$ we must therefore have four cubes meeting around $e$. But the link of $e$ in $X$ is $\Gamma_i$ which contains no cycles of length less than 6 -- a contradiction.

If the transverse intersection point $x$ is the endpoint of the edge $e$ we proceed similarly. We will assume that the link of $e$ is $\Gamma_1$ and $x$ is a vertex of the form $(b_3, a_2, a_3)$. Since there is a flat containing $x$ we see that there is a loop of length $2\pi$ in $\lk(x, \tilde{X})$. We can now take the CAT(0) cone $C_0(x)$ over $\lk(x, \tilde{X})$. Since $\lk(x, \tilde{X}) = \Gamma_1\ast \lk(a_2, \Gamma_2)$ we see by Theorem \ref{joinisproduct} $C_0(x)$ is isometric to $C_0(\Gamma_1)\times C_0(\lk(a_2, \Gamma_2))$. The latter summand is several copies of $\RR_+$ joined at $0$. We can project this plane to $C_0(\Gamma_1)$. Two parallel geodesics have bounded distance under this projection and thus by Theorem \ref{flatstrip} they bound a flat strip. This gives a loop of length $2\pi$ in $\Gamma_1$. A loop of length $2\pi$ corresponds to a circuit of length 4. But, once again, $\Gamma_1$ has no loops of length less than $6$.

Therefore, there are no isometrically embedded flat planes in $\tilde{X}$ and $\pi_1(X)$ is hyperbolic.

\subsection{Euler characteristics}

At this stage in our consideration of these examples, we again have an infinite family of CAT(0) spaces, each having hyperbolic fundamental group. We prove that there are infinitely many non-isomorphic groups in this family by calculating the Euler characteristic of the spaces in terms of the sizeable graphs $\Gamma_1$, $\Gamma_2$ and $\Gamma_3$. We would like to improve this result to show that the non-hyperbolic subgroups that we've constructed are not commensurable; this work is ongoing. 

\begin{prop}
Let $\Gamma_1$, $\Gamma_2$ and $\Gamma_3$ be sizeable graphs, such that $\Gamma_i$ has vertices divided into $A_i\sqcup B_i$ as in Definition \ref{sizeabledef}; also let $E_i$ be the edge set of $\Gamma_i$. We will denote $a_i = |A_i|, b_i = |B_i|$ and $e_i = |E_i|$. The Euler characteristic of the cube complex $X = X(\Gamma_1, \Gamma_2, \Gamma_3)$ is 
\begin{multline*}
a_1a_2a_3 + b_1b_2b_3 + a_1e_2 + a_3e_1 + a_2e_3 + b_2e_1 + b_3e_2 + b_1e_3 \\+ e_1e_2a_3 + e_1a_2e_3 + a_1e_2e_3 + e_1e_2b_3 + e_1b_2e_3 + b_1e_2e_3\\ - a_1a_2e_3 - a_1e_2a_3 - e_1a_2a_3 - b_1b_2e_3 - b_1e_2b_3 - e_1b_2b_3\\ - e_1e_2 - e_2e_3 - e_1e_3 - a_1e_2b_3 - e_1b_2a_3 - b_1a_2e_3 - e_1e_2e_3.
\end{multline*}
\end{prop}
\begin{proof}
We examine the links of vertices. Vertices, edges and 2-cells in $Lk(v,X)$ correspond to edges, squares and cubes adjacent to $v$ in $X$. Since each edge is adjacent to two vertices, each square is adjacent to four vertices and each cube is adjacent to eight vertices. The Euler characteristic can be computed by knowing the number of 0, 1 and 2-cells in the link of each vertex. 

We start by examining the 8 types of vertices and their links. A vertex $v$ in $X$ is of the form $(v_1, v_2, v_3)$; the 8 types correspond to $v_i$ being in $A_i$ or $B_{i-1}$. We will use $N_{v_i}$ to denote the neighbours of $v_i$ in the sizeable graph $\Gamma_i$ or $\Gamma_{i-1}$ and $n_i = N_{v_i}$.

\begin{table}[h!]
\begin{tabular}{ccccc}
($v_1,v_2,v_3$) & Link of vertex & 0-cells in link & 1-cells in link & 2-cells in link \\
\hline
$A_1\times A_2\times A_3$ & $N_{v_1}\ast N_{v_2}\ast N_{v_3}$ & $n_1 + n_2 + n_3$ & $n_1n_2 + n_1n_3 + n_2n_3$ & $n_1n_2n_3$\\
$A_1\times A_2\times B_2$ & $\Gamma_3\ast N_{v_1}$ & $a_3 + b_3 + n_1$ & $e_3 + n_1(a_3 + b_3)$ & $e_3n_1$\\
$A_1\times B_1\times A_3$ & $\Gamma_2\ast N_{v_3}$ & $a_2 + b_2 + n_3$ & $e_2 + n_3(a_2 + b_2)$ & $e_2n_3$\\
$A_1\times B_1\times B_2$ & $\Gamma_3\ast N_{v_3}$ & $a_3 + b_3 + n_3$ & $e_3 + n_3(a_3 + b_3)$ & $e_3n_3$\\
$B_3\times A_2\times A_3$ & $\Gamma_1\ast N_{v_2}$ & $a_1 + b_1 + n_2$ & $e_1 + n_2(a_1 + b_1)$ & $e_1n_2$\\
$B_3\times A_2\times B_2$ & $\Gamma_1\ast N_{v_1}$ & $a_1 + b_1 + n_1$ & $e_1 + n_1(a_1 + b_1)$ & $e_1n_1$\\
$B_3\times B_1\times A_3$ & $\Gamma_2\ast N_{v_2}$ & $a_2 + b_2 + n_2$ & $e_2 + n_2(a_2 + b_2)$ & $e_2n_2$\\
$B_3\times B_1\times B_2$ & $N_{v_1}\ast N_{v_2}\ast N_{v_3}$ & $n_1 + n_2 + n_3$ & $n_1n_2 + n_1n_3 + n_2n_3$ & $n_1n_2n_3$
\end{tabular}
\end{table}

Summing over the appropriate sets gives the desired result. We go through the first 2 cases in detail. 

The number of 0-cells in the links of all vertices which are in $A_1\times A_2\times A_3$ is
$$\sum_{v_1\in A_1}\sum_{v_2\in A_2}\sum_{v_3\in A_3}(n_1 + n_2 + n_3) = (a_1a_2\sum_{v_3\in A_3}n_3) + (a_1a_3\sum_{v_2\in A_2}n_2) + (a_2a_3\sum_{v_1\in A_1}n_1),$$
since each edge of $\Gamma_i$ has one endpoint in $A_i$ we can see that this is equal to $$a_1a_2e_3 + a_1e_2a_3 + e_1a_2a_3.$$

We claim that the number of 1-cells in the links of all vertices which are in $A_1\times A_2\times A_3$ is
$$\sum_{v_1\in A_1}\sum_{v_2\in A_2}\sum_{v_3\in A_3}(n_1n_2 + n_1n_3 + n_2n_3).$$
Using the fact that $\Gamma_i$ is bipartite, this is equal to $$e_1e_2a_3 + e_1a_2e_3 + a_1e_2e_3.$$

Finally, the number of 2-cells in the links of all vertices which are in $A_1\times A_2\times A_3$ is
$$\sum_{v_1\in A_1}\sum_{v_2\in A_2}\sum_{v_3\in A_3}n_1n_2n_3 = e_1e_2e_3.$$
We get similar results for the vertices in $B_3\times B_1\times B_2$.

For vertices in $A_1\times A_2\times B_2$ we get the following: the number of 0-cells across all such vertices is
$$\sum_{v_1\in A_1}\sum_{[v_2, v_3]\in E_2} (n_1 + a_3 + b_3) = a_1e_2a_3 + a_1e_2b_3 + e_1e_2;$$
the number of 1-cells across all such vertices is
$$\sum_{v_1\in A_1}\sum_{[v_2, v_3]\in E_2} (e_3 + n_1(a_3 + b_3)) = a_1e_2e_3 + e_1e_2a_3 + e_1e_2b_3;$$
the number of 2-cells across all such vertices is
$$\sum_{v_1\in A_1}\sum_{[v_2, v_3]\in E_2} e_3n_1 = e_1e_2e_3.$$
Repeating this for the other vertices gives the desired result.
\end{proof}

Using the construction detailed in \cite{kropholler_almost_2017} with $|A^+| = |B^+| = 2$ and $|A^-| = |B^-| = 1$ we can create sizeable graphs where $a_i = 4p = b_i$ and $e_i = 16p$ for all primes $p\geq 5$. The Euler characteristic of $X$ in these examples is $-64p^2(p+6)$ and as such we have an infinite family of non-isomorphic hyperbolic groups each having a finitely presented subgroup not of type $F_3$.

\section{Obstruction to hyperbolisation in dimension at least 4}

Our procedure for hyperbolising products of graphs is special to dimension three. In this section we show that in higher dimensions hyperbolisation via branched covers is not possible in the settings discussed. 

Recall that the process of taking a branched cover of $X$ over a branching locus $C$ is the following:

\begin{enumerate}
\item Take a finite covering $\overline{X\smallsetminus C}$ of $X\smallsetminus C$;
\item Lift the piecewise Euclidean metric locally and consider the induced path metric on $\overline{X\smallsetminus C}$;
\item Take the metric completion $\hat{X}$ of $\overline{X\smallsetminus C}$.
\end{enumerate}

\begin{thm}
If $X$ is a cube complex of dimension $n>3$, and $X$ has a cubical subcomplex isometric to $T^n$, then no non-positively curved branched cover of $X$ is hyperbolic.
\end{thm}

Brady \cite{brady_branched_1999} states that Bestvina proved a similar result about $T^5$. 

We will start by looking at the case of an $n$-torus and this will quickly imply the result. 
Let $C$ be a branching locus for $X = T^n$, let $\overline{X\smallsetminus C}$ be the finite cover in the definition of branched cover and let $Y$ be the completion of this finite cover.

\begin{lem}\label{justadd}
Let $e$ be an open cube contained in $C$ of dim $\leq n-3$ and assume there does not exist a cube $e'$ such that $\bar{e}\subsetneq e'\subsetneq C$. Let $C' = C\smallsetminus e$. Then there is a cover of $T^n\smallsetminus C'$ whose completion is isomorphically isometric to $Y$.
\end{lem}
\begin{proof}
Within $C$, $e$ is contained in no larger cube so each $x\in \mathring{e}$ has a neighbourhood $N(x)$ in $T^n$ such that $N(x)\smallsetminus C$ is homeomorphic to $\RR^n\smallsetminus \RR^{\rm{dim}(e)}$. Then $N(x)\smallsetminus C$ is homotopy equivalent to a sphere of dimension at least 2, so is simply connected. This deleted neighbourhood lifts homeomorphically to the cover and when we complete we are just gluing back in the copy of $\RR^{\rm{dim}(e)}$. This shows that in a neighbourhood of each point of $e$, the branched covering map is an unramified cover. Thus we can remove $e$ from the branching locus without affecting $\hat{X}$. 
\end{proof}

This lemma allows us to assume that the branching locus is a union of cubes of dimension $n-2$. 

\begin{lem}\label{alln-2}
Suppose that $e = e^{n-2}\subset C$ has a free face $e'$ (that is, there is no other cube glued to $e$ along $e'$) and let $C' = C\smallsetminus (int(e)\cup e')$. Then there is a cover of $T^n\smallsetminus C'$ whose completion is isometrically isomorphic to $Y$.
\end{lem}
\begin{proof}
Let $e'$ be a free face of $e$. Each point $x\in \mathring{e'}$ has a neighbourhood $N(x)$ such that $N(x)\smallsetminus C$ is homeomorphic to $\RR^n\smallsetminus (\RR_+\times\RR^{n-4})$ which is simply connected. This deleted neighbourhood lifts to the cover, and when we take completions we are just gluing back in the copy of $\RR_+\times\RR^{n-4}$. This shows that in a neighbourhood of the free face the branched covering map is an unramified covering. 

Let $x$ be an arbitrary point in the interior of $e$ and let $\gamma$ be a geodesic from $x$ to the free face of $e$. The boundary $\partial$ of a regular neighbourhood of $\gamma$ in $T^n$ is homeomorphic to $S^{n-1} = S^{n-1}_-\cup (S^{n-2}\times [0,1])\cup S^{n-1}_+$ where $S_{\pm}^{n-1}$ are hemispheres centred around the endpoints of $\gamma$. Then $(\partial \smallsetminus C)\simeq (S^{n-1}\smallsetminus S^{n-1}_-\cup (S^{n-2}\times [0,1]))\simeq \star$. A deleted neighbourhood of $\gamma$ in $T^n\smallsetminus C$ lifts to the cover. So, again,  the branched covering is unramified in a neighbourhood of $\gamma$, which shows that we can remove $e$ from the branching locus.  
\end{proof}

\begin{lem}\label{nofree}
A locally convex subcomplex of $T^n$ of dimension $n-2$ that has no free faces is a disjoint union of copies of $T^{n-2}$. 
\end{lem}

\begin{proof}
The locally convex subcomplex with no free faces is a closed non-positively curved manifold and defines a class in $H_{n-2}(T^n)$. The link of every vertex in $T^n$ is a copy of $S^{n-1}$ triangulated as $S^0\ast\dots\ast S^0$. After a cubical subdivision (if necessary) we can assume that the cubical neighbourhood of every vertex is $[-1,1]^n$. Each copy of $[-1,1]$ in this product defines a generator of the fundamental group of $T^n$. In the subcomplex each vertex will have link isomorphic to $S^0\ast\dots\ast S^0 = S^{n-3}$, where the $n-2$ copies of $S^0$ give $n-2$ generators showing that this submanifold has fundamental group $\ZZ^{n-2}$ and the manifold is an $(n-2)$-torus. 
\end{proof}

\begin{prop}
Let $b\colon  Y\to T^n$ be a branched cover. If a component $\Delta$ of the branching locus is a copy of $T^{n-2}$, then the restriction of $b$ to the preimage of $\Delta$ is a covering map.
\end{prop}
\begin{proof}

The component $\Delta$ has a neighbourhood homeomorphic to $T^{n-2}\times D^2$. Deleting $\Delta$ we obtain $T^{n-2}\times (D^2\smallsetminus\{0\})$. Any finite cover of $T^{n-2}\times (D^2\smallsetminus\{0\})$ is again homeomorphic to $T^{n-2}\times (D^2\smallsetminus \{0\})$. Taking the metric completion has the effect of adding $T^{n-2}\times\{0\}$. From this we see that the map on the preimage of $\Delta$ is an unramified covering map. 
\end{proof}

If we know that the cover is non-positively curved and the torus is locally convex, then the fundamental group of our cover will have a $\ZZ^2$ subgroup and will not be hyperbolic. 

\begin{prop}
The preimage of $C$ is locally convex and the cover is non-positively curved.  
\end{prop}
\begin{proof}
Looking at a point in $C$, we see a neighbourhood of the form $S^1\ast D^{n-2}$ where $S^1$ has length $2\pi$. The preimage of this in the cover will be several disjoint copies of $S^1\ast D^{n-2}$, where the length of $S^1$ may now be greater than $2\pi$. The copy of $D^{n-2}$ is convex in $S^1\ast D^{n-2}$ and so the embedding of the preimage of $C$ is locally convex. 

The link of a vertex will be copy of $S^1\ast S^{n-3}$, which is a CAT(1) space, as the join of two CAT(1) spaces. Thus the cover is non-positively curved.  
\end{proof}

\begin{thm}
If $X$ is an $n$-dimensional cube complex, $n>3$, which contains an isometric copy of $n$-dimensional torus as a cubical subcomplex, then no non-positively curved branched cover of $X$ is hyperbolic. 
\end{thm}
\begin{proof}
Consider a component of the preimage of $T^n$. This is a branched cover of $T^n$, that will contribute a $\ZZ^2$ subgroup to the fundamental group of the branched covering, which is therefore not hyperbolic. 
\end{proof}

In the case where $X$ is the product of $n\geq 4$ graphs with no vertices of valence 1, every branched cover is CAT(0). By Lemmas \ref{justadd}, \ref{alln-2} and \ref{nofree} we can assume that the branching locus $C$ is a union of $n-2$ cubes. Let $\hat{C}$ be the preimage of $C$ under the branching map $b$. Using the above we can also see that the map $b\colon \hat{C}\to L$ is a covering map.

\begin{lem}[Brady \cite{brady_branched_1999}, Lemma 5.4]
Let $X$ be a non-positively curved cube complex, $C\subset X$ a branching locus, and $v\in C$ a vertex. Then $\lk(v ,C)$ is a full subcomplex of $\lk(v , X)$.
\end{lem}

The following proposition shows that the cover will be non-positively curved.

\begin{prop}
Let $K$ be a non-positively curved cube complex and let $L$ be a branching locus. Let $\hat{L}$ be the preimage of $L$ under the branched covering $b\colon \hat{K}\to K$. If the map $b\colon \hat{L}\to L$ is an unramified covering, then $\hat{K}$ is non-positively curved. 
\end{prop}
\begin{proof}
The branched covering $\hat{K}$ is a cube complex and we must prove that the link of each vertex is a flag complex. Let $v$ be a vertex of $\hat{K}$ there is a derivative map $b_{\lk(v)}\colon \lk(v, \hat{K})\to \lk(b(v), K)$. We must show that the complex $\lk(v, \hat{K})$ is simplicial. If not, there exist simplices $\sigma, \tau\subset \lk(v, \hat{K})$ which are not equal but have the same boundary. Since $\lk(b(v), K)$ is a simplicial complex we can see that $b_{\lk(v)}(\sigma) = b_{\lk(v)}(\tau)$. 

A vertex $w$ in $\sigma$ defines an edge at $v$ in $\hat{K}$. If $b(e)\not\subset L$ then the map $\overline{\lk}(e, \hat{K})\to \overline{\lk}(b(e), K)$ is an isomorphism. Since $\lk(e, \hat{K})$ can be identified with $\lk(w, \lk(v, \hat{K}))$. This implies that the map of links $\lk(v , \hat{K}) \to \lk(b(v), K)$ is locally injective at $w \in \lk(v , \hat{K} )$ contradicting $b_{\lk(v)}(\sigma) = b_{\lk(v)}(\tau)$.

This shows that all the vertices of $\sigma$ are contained in $\hat{L}$. But the map $b\colon \hat{L}\to L$ is a covering map, in particular locally injective, which once again contradicts the fact that $b_{\lk(v)}(\sigma) = b_{\lk(v)}(\tau)$.

Once it is simplicial we see that the simplex in $K$ filling the 1-skeleton of $b_{\lk(v)}(\sigma)$ has a corresponding cube. This cube lifts to the cover and fills the 1-skeleton of $\sigma$ in $\lk(v, \hat{K})$. 
\end{proof}

\appendix

\section{Minimal sizeable graphs -- by Giles Gardam}

In this appendix we determine minimal sizeable graphs: we construct a sizeable graph (as defined in Definition~\ref{sizeabledef}) on 24 vertices (Figure~\ref{fig:sizeable24}), and prove that there exists no such graph on 23 or fewer vertices (Theorem~\ref{thm:sizeable24}).
We also determine the smallest sizeable graphs in certain classes: where the 4 subgraphs are all paths (as in \cite{kropholler_thesis}), all cycles, and all cycles defined `arithmetically' as in \cite{lodha_finiteness_2017}.

\begin{unnumberednotation}
    For our sizeable graphs, we will denote the two bipartitions of the vertex sets as $A = A_0 \sqcup A_1$ and $B = B_0 \sqcup B_1$, rather than using $+$ and $-$ as indices.
    The 4 induced subgraphs $\Gamma(A_s, B_t)$ will be referred to as the \emph{defining subgraphs} of $\Gamma$.
\end{unnumberednotation}

First, let us sketch why constructing sizeable graphs is a delicate matter.
On the large scale, a sizeable graph ``looks like'' a 4-cycle: think of shrinking each $A^s$, $B^t$ to a single vertex, and draw an edge to represent an induced subgraph that is connected.
This runs contrary to containing no 4-cycle, which makes constructing such graphs difficult.

Suppose that each $A^s$ and $B^t$ has $n$ edges.
Connectivity of the defining subgraphs requires that they have average degree approximately 2 (so average degree 4 in the whole graph), since the sum of degrees over the $2n$ vertices must be at least $2(2n-1)$.
\begin{prop}
    Let $\Gamma$ be a random bipartite graph on $A \sqcup B$ with $\abs{A} = \abs{B} = 2n$, constructed by including any edge from $A$ to $B$ independently with probability $\frac{2}{n}$.
    Then the expected number of subgraphs of $\Gamma$ isomorphic to the 4-cycle is $(1-\frac{1}{2n})^2 64$.
\end{prop}

This underlines the difficulty of constructing such graphs.

\begin{proof}
    The number of possible $4$-cycles is ${\binom{2n}{2}}^2$, and the probability of any given $4$-cycle occurring is $\left(\frac{2}{n}\right)^4$.
\end{proof}

\begin{definition}
    We call a graph \emph{arithmetic} if it can be constructed in the following manner.
    The vertices are divided into 4 sets $A_0, A_1, B_0, B_1$ of equal size $n$, identified with the group $\mathbb{Z} / n$.
    The edges are determined by 8 numbers: $h_{s,t}$ and $k_{s,t}$ for $s, t \in \{0, 1\}$.
    The vertex $a_s^i \in A^s$ is joined by an edge to $b_t^{i+h_{s,t}}$ and $b_t^{i+k_{s,t}}$ in $B^t$.
\end{definition}

Lodha constructed arithmetic sizeable graphs for $n = 11$, that is, on $44$ vertices.
The graph in Figure~\ref{fig:arithmetic_example} is an example for $n = 9$, with the values
\begin{align*}
    h_{0,0} &=  0 & h_{0,1} &= 0 & h_{1,0} &= 0 & h_{1,1} &= -1 \\
    k_{0,0} &= -1 & k_{0,1} &= 2 & k_{1,0} &= 2 & k_{1,1} &= -2 \\
\end{align*}
(the difference in sign is due to orientation of the edges from $A_0$ and towards $A_1$, used for convenience in the proof of Theorem~\ref{thm:arithmetic}).
\begin{thm}
    \label{thm:arithmetic}
    The smallest number of vertices for an arithmetic sizeable graph is 36.
\end{thm}

\begin{proof}
    Our example shows that 36 is possible; we now show that less than 36 is not.
    Consider an arithmetic graph.
    For any vertex $a_0 \in A_0$, there are 8 paths of length 2 from $a_0$ to a vertex $a_1 \in A_1$: such a path can go via $B_0$ or $B_1$, and there are two choices of a neighbour $b$ of $a_0$ in each corresponding $B^t$, and $b$ then has 2 neighbours in $A_1$.
    If any of these paths had the same endpoint, this would give a 4-cycle.
    Thus $A_1$ must have at least 8 vertices, so the whole graph must have at least $32$.

    Suppose for the sake of contradiction that we had an arithmetic sizeable graph on $32$ vertices, that is, for $n = 8$.
    After a change of coordinates, we can assume that $h_{0,0} = h_{1,0} = h_{0,1} = 0$.
    To be precise, we relabel the vertex $b_0^i$ with $b_0^{i-h_{0,0}}$, $b_1^i$ with $b_1^{i-h_{0,1}}$, and $a_1$ with $a_1^{i-h_{0,0}+h_{1,0}}$ and adjust the values of $h_{s,t}$ and $k_{s,t}$ accordingly afterwards.
    For notational convenience, as indicated in Figure~\ref{fig:arithmetic_generic} we let $a = k_{0,0}$, $b = -{k_{1,0}}$, $c = k_{0,1}$, $d = -h_{1,1}$ and $e = -k_{1,1}$.
    \begin{figure}
        \centering
        \begin{subfigure}[t]{0.45\textwidth}
            \centering
            \def\svgwidth{\textwidth}
            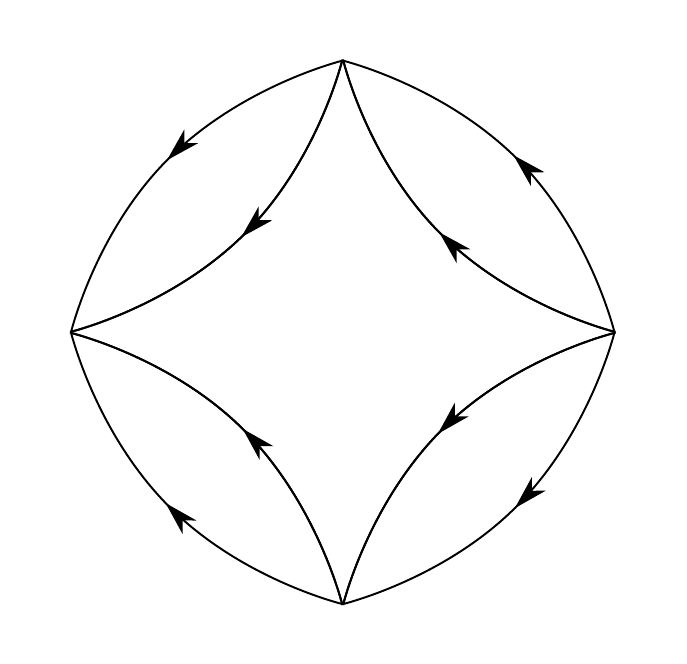
            \caption{A generic arithmetic graph}
            \label{fig:arithmetic_generic}
        \end{subfigure}
        \hfill
        \begin{subfigure}[t]{0.45\textwidth}
            \centering
            \def\svgwidth{\textwidth}
            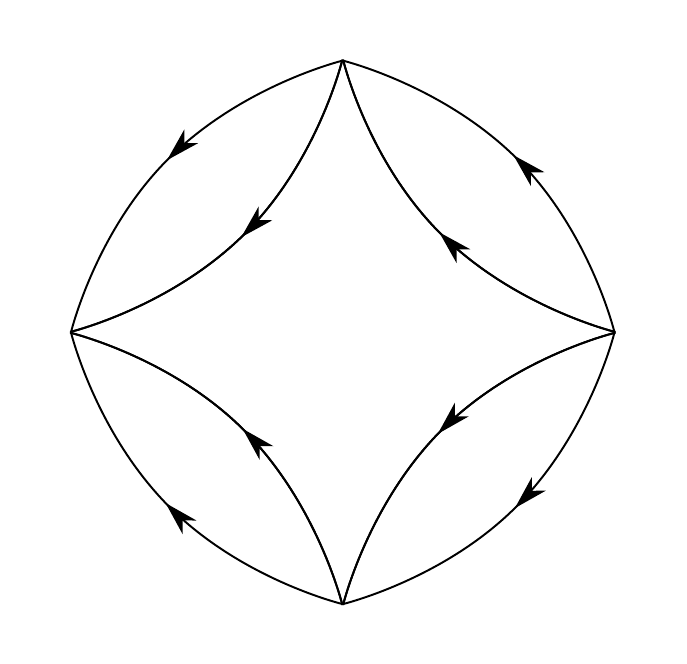
            \caption{A sizeable arithmetic graph for $n = 9$}
            \label{fig:arithmetic_example}
        \end{subfigure}
        \caption{Arithmetic graphs}
    \end{figure}

    The 8 paths of length 2 from $a_0^0$ to $A_1$ end at the vertices $a_1^i$ for $i = 0, a, b, a+b, d, c+d, e, c+e$, and thus these are a permutation of $0, \dots, 7 \pmod{8}$.
    Summing these 8 numbers, we have \[
        2(a + b + c + d + e) \equiv 28 \equiv 4 \pmod{8}
    \]
    However, for the 4 defining subgraphs to be connected (and not to split as 2 or more disjoint cycles) we need $a$, $b$ and $c$ to be odd, and $d$ and $e$ to be of opposite parity, so without loss of generality assume $d$ is even and $e$ is odd.
    Now the 4 endpoints $a_1^i$ for $i = a$, $b$, $c + d$ and $e$ must be a permutation of $1, 3, 5, 7$, so that \[
        a + b + c + d + e \equiv 1 + 3 + 5 + 7 \equiv 0 \pmod{8}.
    \] This is a contradiction.
\end{proof}

For the lower bound of 32, all we used was the assumption that in each of the 4 defining subgraphs, every vertex has degree 2.
For instance, this is true if the subgraphs are (not necessarily arithmetically defined) cycles.
By exhaustive computational search, we have determined there is no sizeable graph of 32 vertices of this type either, so a minimal sizeable graph with all defining subgraphs cycles is our arithmetic example on 36 vertices.

An immediate observation is that we can assume the defining subgraphs of any sizeable graph are trees: the definition requires only that they be connected, but whenever one is not a tree, we could remove an edge from a cycle and obtain a sizeable graph on the same number of vertices (since we clearly will not have introduced a 4-cycle, nor affected connectivity of subgraphs) but with one less edge.
So it is natural to ask how small we can make a sizeable graph where the 4 defining subgraphs are not cycles, but rather paths.
In fact, intuition suggests that one could well construct sizeable graphs of the absolute minimum number of vertices in this way, because one anticipates that the branching of any tree that is not a path creates central, highly connected vertices that would serve to create 4-cycles.

However, with the 4 defining subgraphs as paths we cannot achieve the minimum of 24, or even very close.
Most of this gap is explained by the fact that having the total number of edges in each subgraph more \emph{evenly distributed} over the degrees of the vertices (which have degree 2 except for in 2 cases) means that every vertex in $B$ joins sufficiently many pairs in $A_0 \times A_1$ that unless there are 28 vertices, creating a 4-cycle is unavoidable by the pigeonhole principle.
We now prove this.

\begin{prop}
    \label{prop:path28}
    Any sizeable graph with all 4 defining subgraphs a path has at least 28 vertices.
\end{prop}

\begin{proof}
    Let the $n = \abs{B_0}$ vertices in $B_0$ be labelled $1, 2, \dots n$ and let $d_i$ denote the degree of vertex $i$ in the subgraph $\Gamma(A_0, B_0)$ and let $e_i$ be its degree in $\Gamma(A_1, B_0)$.
    Since the defining subgraphs are paths, every $d_i$ and every $e_i$ is either $1$ or $2$.
    The sum of degrees over one side of a bipartite graph is the number of edges, so $\sum_{i=1}^n d_i = \abs{A_0} + \abs{B_0} - 1$ and it follows that the number of $d_i$ that are equal to $1$ is $\abs{B_0} + 1 - \abs{A_0}$ (which is either 0, 1, or 2, noting that in a bipartite path graph the two sides have size differing by at most 1).
    Similarly, precisely $\abs{B_0} + 1 - \abs{A_1}$ of the $e_i$ are equal to $1$.

    The number of pairs in $A_0 \times A_1$ that are joined by a vertex in $B_0$ is $\sum_{i=1}^n d_i e_i$.
    By the rearrangement inequality, this is bounded below by the value for when we never have $d_i = e_i = 1$, that is, when there are $\abs{B_0} + 1 - \abs{A_0}$ pairs $(d_i, e_i) = (2, 1)$, $\abs{B_0} + 1 - \abs{A_1}$ pairs $(d_i, e_i) = (1, 2)$, and the remaining pairs are $(2, 2)$ (if such an arrangement is not possible due to there being too many 1's, the following bound still holds).
    Thus we have \[
        \sum_{i=1}^n d_i e_i \geq 4n - 2 ( 2\abs{B_0} + 2 - \abs{A_0} - \abs{A_1}) = 2 (\abs{A_0} + \abs{A_1} - 2 )
    \]
    We have the same lower bound on the number of pairs in $A_0 \times A_1$ joined by a vertex in $B_1$, so since these pairs must be distinct, we have \[
        \abs{A_0} \abs{A_1} \geq 4 (\abs{A_0} + \abs{A_1} - 2)
    \] which gives $(\abs{A_0} - 4) (\abs{A_1} - 4) \geq 8$.
    This is possible (for natural numbers) only if $\abs{A_0} + \abs{A_1} \geq 14$.
    We similarly conclude $\abs{B_0} + \abs{B_1} \geq 14$ so the graph has at least 28 vertices in total.
\end{proof}

\begin{rem}
    By exhaustive computational search, we know that actually the minimum size we can attain is 31 vertices.
\end{rem}

\begin{thm}
    \label{thm:sizeable24}
    The minimal number of vertices of a sizeable graph is 24.
\end{thm}

\begin{figure}
\centering
\begin{subfigure}[b]{.7\textwidth}
\includegraphics[width=\textwidth]{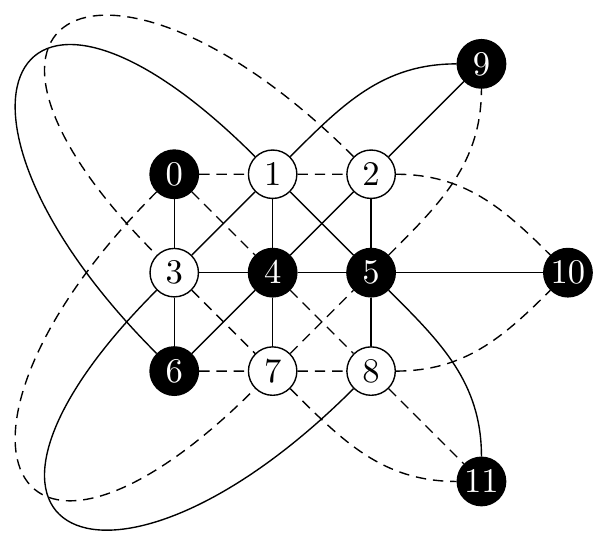}
\vspace{6ex} 
\end{subfigure}\quad
\begin{subfigure}[b]{.23\textwidth}

\includegraphics[width=\textwidth]{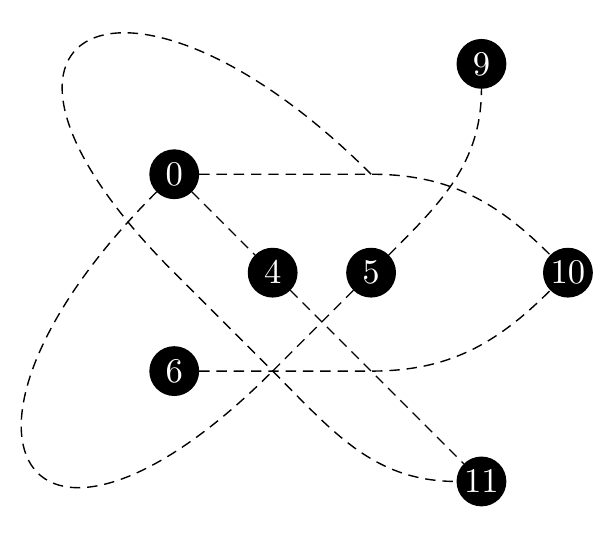}
\includegraphics[width=\textwidth]{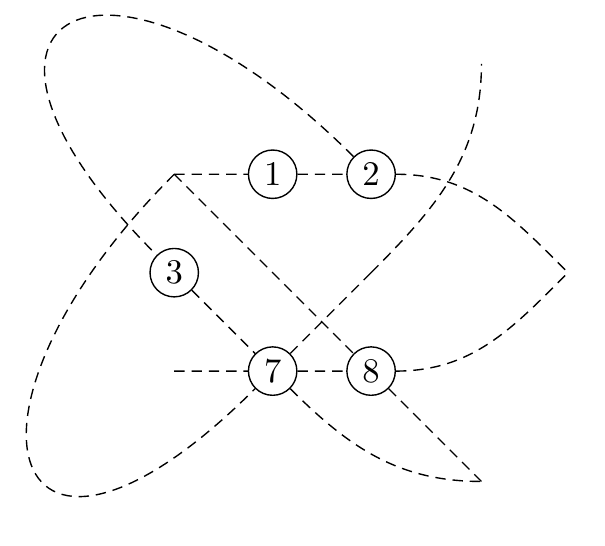}
\includegraphics[width=\textwidth]{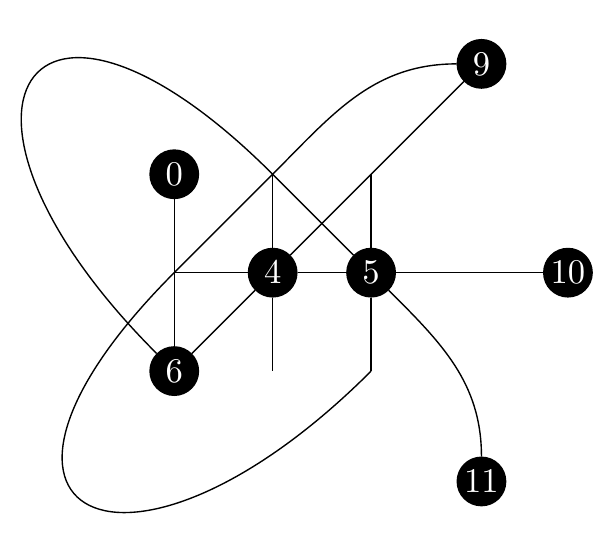}
\includegraphics[width=\textwidth]{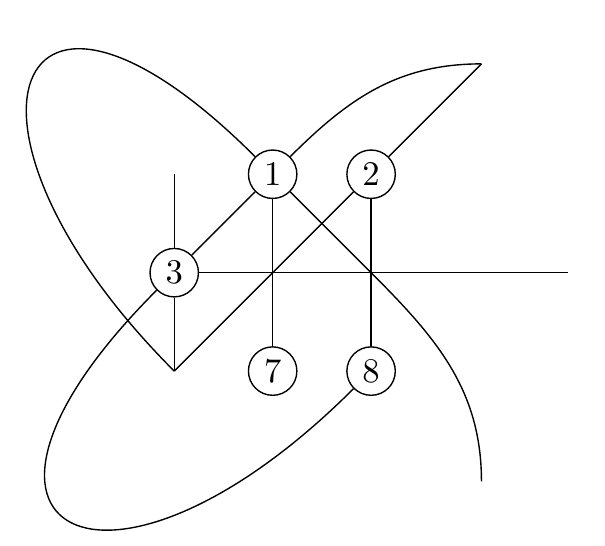}

\end{subfigure}
\caption{A sizeable graph on 24 vertices, depicted as a subgraph of the incidence graph for the projective plane of order 3 (Figure~\ref{fig:full_proj_plane}), shown together with the 4 induced subgraphs}
\label{fig:sizeable24}
\end{figure}

With no topological assumptions on the subgraphs (such as being cycles or paths) the best lower bound we can get with the techniques used to prove Proposition~\ref{prop:path28} is 22, which is close to the actual minimum of 24.

\begin{prop}
    Every sizeable graph has at least 22 vertices.
\end{prop}

\begin{proof}
    As in the proof of the above proposition, let the degrees of vertices of $B_0$ be $d_i$ and $e_i$.
    Connectivity requires that $d_i, e_i \geq 1$ and thus $(d_i - 1)(e_i - 1) \geq 0$, so $d_i e_i \geq d_i + e_i - 1$.
    Thus $\sum_{i=1}^n d_i e_i \geq \sum_{i=1}^n d_i + \sum_{i=1}^n e_i - \abs{B_0} = (\abs{A_0} + \abs{B_0} - 1) + (\abs{A_1} + \abs{B_0} - 1) - \abs{B_0} = \abs{A_0} + \abs{A_1} + \abs{B_0} - 2$.

    We have a similar bound for $B_1$, and thus considering the total number of pairs in $A_0 \times A_1$ at distance 2 in the graph metric, we see that \[
        \abs{A_0} \abs{A_1} \geq \abs{A_0} + \abs{A_1} - 4 + (\abs{A_0} + \abs{A_1} + \abs{B_0} + \abs{B_1})
    \] We can assume without loss of generality that $\abs{B_0} + \abs{B_1} \geq \abs{A_0} + \abs{A_1}$, so after substituting into the above inequality and factorizing we have \[
        (\abs{A_0} - 3) (\abs{A_1} - 3) \geq 5
    \]
    Since $\abs{A_0}, \abs{A_1}$ are positive integers this implies that $\abs{A_0} + \abs{A_1} \geq 11$, so the graph has at least 22 vertices.
\end{proof}

While we can use such a pleasantly simple argument to get a lower bound of 22, to improve the lower bound to 24 we now apply results in extremal graph theory, concerning the \emph{Zarankiewicz problem}.
To do this, we weaken the requirements on our graphs somewhat: we forget the bipartitions of $A$ into $A_0$ and $A_1$ and of $B$ into $B_0$ and $B_1$, and thus do not ask for connectivity of these subgraphs, but rather ask simply that there be at least as many edges as connectivity would require, which is $2N-4$ for a sizeable graph on $N$ vertices.

\begin{definition}[{\cite[Definition 1.2]{zarankiewicz}}]
    A bipartite graph $G = (A, B; E)$ is called $K_{s,t}$-free if it does not contain $s$ vertices in $A$ and $t$ vertices in $B$ that span a subgraph isomorphic to the complete bipartite graph $K_{s,t}$.
    The maximum number of edges that a $K_{s,t}$-free bipartite graph of size $(m,n)$ may have is the \emph{Zarankiewicz number} $Z_{s,t}(m,n)$.
\end{definition}

Even for $(s,t) = (2,2)$ -- the case we are interested in, corresponding to having no 4-cycles -- not all Zarankiewicz numbers are known exactly, but they are known for the range of values relevant to us \cite[Table 1]{zarankiewicz}.
This smallest value of $m+n$ for which $Z_{2,2}(m, n) \geq 2(m+n) - 4$ holds is $23$: $Z_{2,2}(11,12) = 42$ (and is the only possibility with $m+n = 23$ up to swapping $m$ and $n$), whereas for $m+n = 22$ we have
\begin{center}
\begin{tabular}{r | c c c c c}
    $(m,n)$ & $(7,15)$ & $(8,14)$ & $(9,13)$ & $(10,12)$ & $(11, 11)$ \\
    \hline
    $Z_{2,2}(m,n)$ & $33$ & $35$ & $37$ & $39$ & $39$
\end{tabular}
\end{center}

\begin{cor}
    \label{cor:at_least_23}
    A sizeable graph has at least 23 vertices.
\end{cor}

A sizeable graph on 23 vertices lies somehow just beyond the cusp of what is possible.
This makes the situation incredibly constrained; we suppose that we were to have a sizeable graph on 23 vertices and progressively determine more and more of its structure, until we zero in on a specific contradiction.
Up to symmetry, there is at each step of the argument only one way for the graph to develop.

\begin{figure*}
    \centering
    \begin{tikzpicture}[
        every node/.style={circle, draw=black, fill=white, text=black, inner sep=0, minimum size=14, thin},
    ]
\coordinate (mySW) at (-2.7, -3.3);
\coordinate (myNE) at (3.3, 2.7);
\clip (mySW) rectangle (myNE);
\useasboundingbox (mySW) rectangle (myNE);

    \path (-1,-1) node (6) {6};
    \path (-1,0) node (3) {3};
    \path (-1,1) node (0) {0};
    \path (0,-1) node (7) {7};
    \path (0,0) node (4) {4};
    \path (0,1) node (1) {1};
    \path (1,-1) node (8) {8};
    \path (1,0) node (5) {5};
    \path (1,1) node (2) {2};
    \path (45:3) node (9) {9};
    \path (0:3) node (10) {10};
    \path (-45:3) node (11) {11};
    \path (-90:3) node (12) {12};

    \draw (0) to (1) to (2) to[out=0,in=135] (10);
    \draw (3) to (4) to (5) to (10);
    \draw (6) to (7) to (8) to[out=0,in=-135] (10);

    \draw (0) to (3) to (6) to[out=-90,in=135] (12);
    \draw (1) to (4) to (7) to (12);
    \draw (2) to (5) to (8) to[out=-90,in=45] (12);

    \draw (6) to[out=135,in=135, distance=100] (1) (1) to (5) (5) to[out=-45,in=90] (11);
    \draw (0) to (4) to (8) to (11);
    \draw (2) to[out=135,in=135, distance=100] (3) (3) to (7) (7) to[out=-45,in=180] (11);

    \draw (0) to[out=-135,in=-135, distance=100] (7) (7) to (5) (5) to[out=45,in=-90] (9);
    \draw (6) to (4) to (2) to (9);
    \draw (8) to[out=-135,in=-135, distance=100] (3) (3) to (1) (1) to[out=45,in=180] (9);

    \draw[bend left=17] (9) to (10) (10) to (11) (11) to (12);
    \end{tikzpicture}
    \caption{The projective plane of order 3 with 13 points and 13 lines. Each pair of points defines a line, and each pair of lines intersects in one point.}
    \label{fig:full_proj_plane}
\end{figure*}
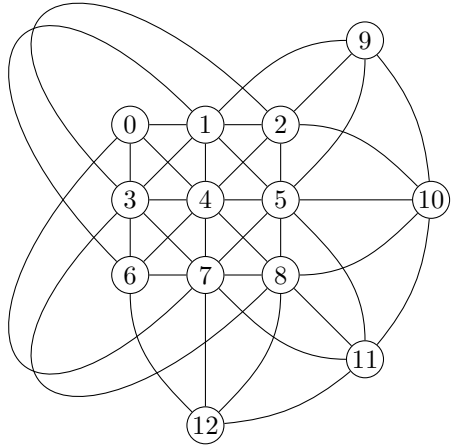

\begin{figure*}
    \centering
    \begin{subfigure}[b]{0.45\textwidth}
        \centering
        \begin{tikzpicture}[
            every node/.style={circle, draw=black, fill=white, text=black, inner sep=0, minimum size=14, thin},
            A1/.style={every node},
            A2/.style={every node},
        ]
\coordinate (mySW) at (-2.7, -3.3);
\coordinate (myNE) at (3.3, 2.7);
\clip (mySW) rectangle (myNE);
\useasboundingbox (mySW) rectangle (myNE);

\path (-1,-1) node[A1] (6) {6};
\path (-1,0) node[A1] (3) {3};
\path (-1,1) node[A1] (0) {0};
\path (0,-1) node[A2] (7) {7};
\path (0,0) node[A2] (4) {4};
\path (0,1) node[A1] (1) {1};
\path (1,-1) node[A2] (8) {8};
\path (1,0) node[A2] (5) {5};
\path (1,1) node[A1] (2) {2};
\path (0:3) node[A2] (10) {10};
\path (-90:3) node[A2] (12) {12};

\draw (0) to (1) to (2) to[out=0,in=135] (10);
\draw (3) to (4) to (5) to (10);
\draw (6) to (7) to (8) to[out=0,in=-135] (10);

\draw (0) to (3) to (6) to[out=-90,in=135] (12);
\draw (1) to (4) to (7) to (12);
\draw (2) to (5) to (8) to[out=-90,in=45] (12);

\draw (6) to[out=135,in=135, distance=100] (1) (1) to (5);
\draw (0) to (4) to (8);
\draw (2) to[out=135,in=135, distance=100] (3) (3) to (7);

\draw (0) to[out=-135,in=-135, distance=100] (7) (7) to (5);
\draw (6) to (4) to (2);
\draw (8) to[out=-135,in=-135, distance=100] (3) (3) to (1);
        \end{tikzpicture}
        \caption{The subgraph of the incidence graph of the projective plane}
        \label{fig:graph23blank}
    \end{subfigure}
    \hfill
    \begin{subfigure}[b]{0.45\textwidth}
        \centering
        \begin{tikzpicture}[
            every node/.style={circle, draw=black, fill=white, text=black, inner sep=0, minimum size=14, thin},
            A1/.style={circle, draw=black, fill=white,text=black, inner sep=0, minimum size=14, thin},
            A2/.style={circle, draw=black, fill=black, text=white, inner sep=0, minimum size=14, thin},
        ]
\coordinate (mySW) at (-2.7, -3.3);
\coordinate (myNE) at (3.3, 2.7);
\clip (mySW) rectangle (myNE);
\useasboundingbox (mySW) rectangle (myNE);

\path (-1,-1) node[A1] (6) {6};
\path (-1,0) node[A1] (3) {3};
\path (-1,1) node[A1] (0) {0};
\path (0,-1) node[A2] (7) {7};
\path (0,0) node[A2] (4) {4};
\path (0,1) node[A1] (1) {1};
\path (1,-1) node[A2] (8) {8};
\path (1,0) node[A2] (5) {5};
\path (1,1) node[A1] (2) {2};
\path (0:3) node[A2] (10) {10};
\path (-90:3) node[A2] (12) {12};

\draw (0) to (1) to (2) to[out=0,in=135] (10);
\draw (3) to (4) to (5) to (10);
\draw (6) to (7) to (8) to[out=0,in=-135] (10);

\draw (0) to (3) to (6) to[out=-90,in=135] (12);
\draw (1) to (4) to (7) to (12);
\draw (2) to (5) to (8) to[out=-90,in=45] (12);

\draw (6) to[out=135,in=135, distance=100] (1) (1) to (5);
\draw (0) to (4) to (8);
\draw (2) to[out=135,in=135, distance=100] (3) (3) to (7);

\draw (0) to[out=-135,in=-135, distance=100] (7) (7) to (5);
\draw (6) to (4) to (2);
\draw (8) to[out=-135,in=-135, distance=100] (3) (3) to (1);
        \end{tikzpicture}
        \caption{With the unique bipartition on points}
        \label{fig:graph23bipartite}
    \end{subfigure}
    \vskip\baselineskip
    \begin{subfigure}[b]{0.45\textwidth}
        \centering
        \begin{tikzpicture}[
            every node/.style={circle, draw=black, fill=white, text=black, inner sep=0, minimum size=14, thin},
            A1/.style={circle, draw=black, fill=white,text=black, inner sep=0, minimum size=14, thin},
            A2/.style={circle, draw=black, fill=black, text=white, inner sep=0, minimum size=14, thin},
        ]
            \coordinate (mySW) at (-2.0, -3.3);
\coordinate (myNE) at (3.3, 0.3);
\clip (mySW) rectangle (myNE);
\useasboundingbox (mySW) rectangle (myNE);

\path (0,-1) node[A2] (7) {7};
\path (0,0) node[A2] (4) {4};
\path (1,-1) node[A2] (8) {8};
\path (1,0) node[A2] (5) {5};
\path (0:3) node[A2] (10) {10};
\path (-90:3) node[A2] (12) {12};

\draw (4) to (5) to (10);
\draw (7) to (8) to[out=0,in=-135] (10);

\draw (4) to (7) to (12);
\draw (5) to (8) to[out=-90,in=45] (12);

\draw (4) to (8);

\draw (7) to (5);
        \end{tikzpicture}
        \caption{Two subgraphs that must be connected}
        \label{fig:graph12}
    \end{subfigure}
    \hfill
    \begin{subfigure}[b]{0.45\textwidth}
        \centering
        \begin{tikzpicture}[
            every node/.style={circle, draw=black, fill=white, text=black, inner sep=0, minimum size=14, thin},
            A1/.style={circle, draw=black, fill=white,text=black, inner sep=0, minimum size=14, thin},
            A2/.style={circle, draw=black, fill=black, text=white, inner sep=0, minimum size=14, thin},
            B1/.style={},
            B2/.style={densely dashed},
        ]
            \coordinate (mySW) at (-2.0, -3.3);
\coordinate (myNE) at (3.3, 0.3);
\clip (mySW) rectangle (myNE);
\useasboundingbox (mySW) rectangle (myNE);

\path (0,-1) node[A2] (7) {7};
\path (0,0) node[A2] (4) {4};
\path (1,-1) node[A2] (8) {8};
\path (1,0) node[A2] (5) {5};

\draw[B1] (4) to (5);
\draw[B2] (7) to (8);

\draw[B1] (4) to (7);
\draw[B2] (5) to (8);

\draw[B1] (4) to (8);
\draw[B2] (7) to (5);
        \end{tikzpicture}
        \caption{The final contradiction}
        \label{fig:graph10}
    \end{subfigure}
    \caption{The stages of the proof by contradiction that there is no sizeable graph on 23 vertices}
    \label{fig:proj_plane_subgraphs}
\end{figure*}
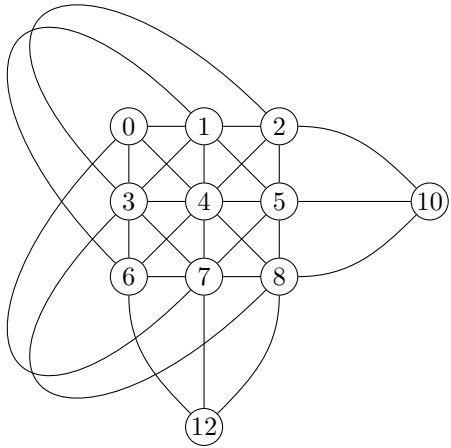
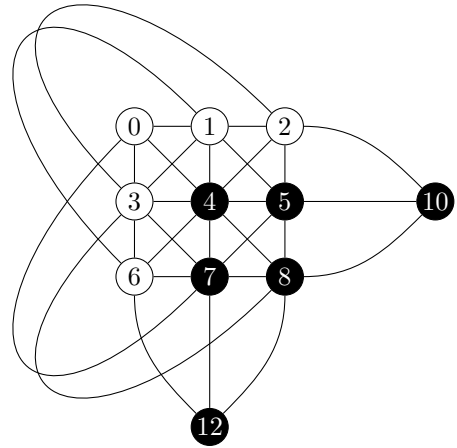
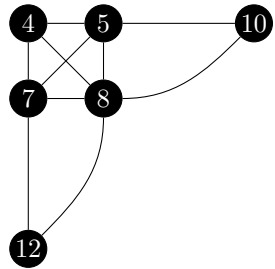
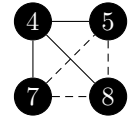

The incidence graph of the projective plane of order 3, Figure~\ref{fig:full_proj_plane}, is a natural candidate for a small sizeable graph.
The projective planes are incidence geometries that are optimal in various senses.
Of particular importance for us is that they are bipartite and have no 4-cycles: a 4-cycle would correspond to a pair of lines that intersected in two distinct points.
The projective plane of order 3 then presents itself since its points and lines all have $3+1=4$ incidences, so that the total number of edges is larger than the minimum needed for connectivity of the 4 induced subgraphs.
Exhaustive computation revealed that the incidence graph of the projective plane of order 3 cannot give a sizeable graph: no bipartition of $A$ and of $B$ will give 4 connected subgraphs.
However, our minimal sizeable graph is a subgraph of this incidence graph; given a sizeable graph on $n$ vertices one cannot necessarily extend it to a sizeable graph on $n+1$ vertices.

We reproduce here the relevant part of Corollary 3.19 of \cite{zarankiewicz}.
This is a stability result, saying that graphs close to having the desired combinatorial prorties of the incidence graph of the projective plane can be embedded in it.
\begin{cor}
    \label{cor:key_cor}
    Let $c \in \mathbb{N}$. Then \[
        Z_{2,2} (n^2 + c, n^2 + n) \leq n^2(n+1) + cn
    \]
    Moreover, if $c \leq n+1$, then graphs reaching the bound can be embedded into a projective plane of order $n$.
\end{cor}

Now we are equipped to prove that no sizeable graph has fewer than 24 vertices, giving the main theorem of this appendix.

\begin{proof}[Proof of Theorem~\ref{thm:sizeable24}]
By Corollary~\ref{cor:at_least_23} and the example of Figure~\ref{fig:sizeable24}, it only remains to rule out a sizeable graph on 23 vertices.
For $n=3$ and $c=2$, Corollary~\ref{cor:key_cor} tells us that any sizeable graph on 23 vertices, which must achieve the bound on $Z_{2,2}(11,12)$ (by tabulated values), is embedded in the incidence graph of the projective plane of order $3$.
This incidence geometry has 13 points and 13 lines.
Since moving between the incidence geometry and the incidence graph it defines can be a source of confusion, we emphasize the following: points and lines refer to the incidence geometry, whereas vertices and edges refer to the corresponding graph.
By duality (there is an automorphism of the incidence geometry that interchanges points and lines), we can assume that it is formed by removing 2 points and 1 line.
Every point is incident to 4 lines, and vice versa, so the total number of incidences is 52.
After removing two points we will have removed 8 incidences, leaving 44.
Thus the line we remove must in fact be the unique line on which the two points lie: this would mean its removal only affects the 2 remaining points on the line, leaving 42 incidences, whereas the removal of any other line would destroy at least 3 incidences.
Since the group of automorphisms of the projective plane is 2-transitive on the set of points (since $GL_3 \mathbb{F}_3$ is 2-transitive on lines in $\mathbb{F}_3^3$, being transitive on linearly independent tuples), we have only one possibility for the bipartite graph on $A$, the points, and $B$, the lines.
For concreteness, we remove the points labelled 9 and 11 in our standard figure of the projective plane of order 3, Figure~\ref{fig:full_proj_plane}, and the line they define, as shown in Figure~\ref{fig:graph23blank}.
The incidence structure left is an affine plane on the 9 points labelled from 0 to 8, together with two points at infinity, labelled by 10 and 12.

Now all that remains is to show that we cannot give $A$ and $B$ the requisite bipartite structures with 4 connected subgraphs.
First we show that there is only one possibility for the bipartition of the points.
For the induced subgraphs to be connected, we of course need every line to have degree at least 1 in the two subgraphs in which it occurs; this means precisely that each line must contain at least one point both from $A_0$ and from $A_1$.
This has a very strong consequence, as a small counting argument now reveals.
Since there can only be one line incident to a given pair of points, the total number of triples $(a_0, a_1, b)$ where points $a_0 \in A_0$ and $a_1 \in A_1$ both lie on the line $b$ is at most $\abs{A_0} \cdot \abs{A_1}$.
Counting over the lines first, we see that a line containing $k$ points creates at least $k-1$ such tuples, with equality if and only if all but 1 point on the line are in the same $A_i$.
There are 6 lines containins 4 points and 6 lines containing 3, so this gives at least $6 \times (4-1) + 6 \times (3-1) = 30$ such triples.
On the other hand, since $\abs{A} = \abs{A_0} + \abs{A_1} = 11$, the maximum value possible for $\abs{A_0} \cdot \abs{A_1}$ is $30$.
Thus without loss of generality, $\abs{A_0} = 5$ and $\abs{A_1} = 6$, and moreover no line can contain 2 points of $A_0$ and 2 points of $A_1$.

Consider now the points ``at infinity'' labelled 10 and 12.
A line passing through them contains either 0 or 2 points in the same side of the bipartition of points.
Summing over the three lines, we see that of the 9 points (labels 0 to 8) on the affine plane, an even number is contained in the same side of the partition.
Thus the two points at infinity must be in the same side as each other, and this must be $A_1$, of size 6, since it thus contains an even number of points.
Thus the three lines incident to one of these points at infinity will comprise 1 line that is entirely $A_0$ inside the affine plane, and 2 lines that have only 1 point in $A_0$.
The projective plane with the two points at infinity added admits symmetries permuting the 3 lines incident to any point at infinity, so without loss of generality, we can suppose that the lines $\{0, 1, 2, 10\}$ and $\{0, 3, 6, 12\}$ are the two relevant lines which, in the affine plane, are entirely in $A_0$.
That is, up to symmetry, there is only one possible way of forming the bipartition $A = A_0 \sqcup A_1$, which is indicated in Figure~\ref{fig:graph23bipartite}, with $A_0$ in white and $A_1$ in black.

We now consider the two induced subgraphs that involve $A_1$.
Any line that only contains one point from $A_1$ can be removed, since it will correspond to a leaf vertex in the corresponding bipartite graph, and its removal will not affect connectivity.
This leaves us with Figure~\ref{fig:graph12}.
The points 10 and 12 at infinity are now of degree 2 in the bipartite graph, so they must have degree 1 in two subgraphs.
So they will be leaves, and we can safely remove them, retaining the condition that each of them forces the two lines incident to it to be in different sides of $B = B_0 \sqcup B_1$.
The incidence geometry now has 4 points, with each pair of points defining a line (it is in fact the affine plane of order 2).
There are two pairs of non-intersecting lines must go in different sides of the partition (the condition we retained when removing the points at infinity); this forces the same of the third pair, since otherwise one side has only 2 lines and the corresponding graph has 6 vertices but only 4 edges.
Thus, up to symmetry, we have a bipartition of the lines as indicated in Figure~\ref{fig:graph10}.
The graph corresponding to the dashed lines is not connected, so we have reached the desired contradiction.
\end{proof}

\bibliographystyle{plain}
\bibliography{bib,otherref,minimal-sizeable-graphs}

\end{document}